\documentclass[11pt, aos,preprint]{imsart}
\usepackage[utf8]{inputenc}
\RequirePackage[OT1]{fontenc}
\usepackage[margin=1.4in]{geometry}
\usepackage[dvips]{graphics}
\DeclareGraphicsExtensions{.eps.gz,.eps,.epsi.gz,.epsi,.ps,.ps.gz}
\usepackage{epsfig}
\usepackage{color}
\usepackage{latexsym}
\usepackage{graphicx}
\usepackage{amsfonts}
\usepackage{amssymb}
\usepackage{mathrsfs}
\usepackage{verbatim}
\usepackage{amsmath}
\usepackage{amsthm}
\usepackage{latexsym}
\usepackage{mathrsfs}
\usepackage{amsfonts}
\usepackage{color}
\usepackage{import}
\usepackage{appendix}
\usepackage{hyperref}
\hypersetup{
    unicode=false,          
    pdffitwindow=false,     
    pdfstartview={FitH},    
    pdftitle={My title},    
    pdfauthor={Author},     
    pdfsubject={Subject},   
    pdfcreator={Creator},   
    pdfproducer={Producer}, 
    pdfkeywords={keyword1} {key2} {key3}, 
    pdfnewwindow=true,      
    colorlinks=true,       
    linkcolor=blue,          
    citecolor=blue,        
    urlcolor=blue       
}
\RequirePackage[authoryear]{natbib}

\arxiv{arXiv:0000.0000}

\startlocaldefs
\numberwithin{equation}{section}
\theoremstyle{plain}

\newtheorem*{theorem*}{Theorem}
\newtheorem{theorem}{Theorem}
\newtheorem{lemma}{Lemma}
\newtheorem{proposition}{Proposition}
\newtheorem{corollary}{Corollary}

\theoremstyle{remark}
\newtheorem{remark}{Remark}

\def\argmax{\mathop{\rm arg\, max}}

\def\Re{\mathop{{\rm I}\kern-.2em\hbox{\rm R}}\nolimits}

\def\trace{\hbox{trace}}
\def\mathbold{\boldsymbol}

\def\bP{\mathbold{P}}
\def\hbP{{\widehat{\mathbold{P}}}}

\def\matA{{\rm \textbf{A}}}

\def\matD{{\rm \textbf{D}}}\def\scrF{{\mathscr F}}\def\bbE{\mathbb{E}}\def\bbEn{\mathbb{E}_{n}}

\def\matG{{\rm \textbf{G}}}
\def\matH{{\rm \textbf{H}}}\def\hbar{\overline{h}}\def\matM{{\rm \textbf{M}}}
\def\calO{{\cal O}}\def\bbP{\mathbb{P}}\def\matR{{\rm \textbf{R}}}\def\bbR{\mathbb{R}}\def\bbS{\mathbb{S}}\def\matT{{\rm \textbf{T}}}\def\matU{{\rm \textbf{U}}}\def\bu{\mathbold{u}}
\def\matW{{\rm \textbf{W}}}\def\bv{\mathbold{v}}\def\matV{{\rm \textbf{V}}}\def\bv{\mathbold{v}}\def\bX{\mathbold{X}}\def\bx{\mathbold{x}}\def\matX{{\rm \textbf{X}}}
\def\by{\mathbold{y}}\def\bY{\mathbold{Y}}\def\matY{{\rm \textbf{Y}}}
\def\matZ{{\rm \textbf{Z}}}
\def\bzero{\mathbold{0}}

\def\bDelta{\mathbold{\Delta}}\def\epsa{\epsilon}
\def\btheta{\mathbold{\theta}}
\def\hbtheta{{\widehat{\btheta}}}\def\bSigma{\mathbold{\Sigma}}
\def\Sigmahat{\widehat{\Sigma}}
\def\hbSigma{{\widehat{\bSigma}}}\def\tbSigma{{\widetilde{\bSigma}}}
\def\htau{\widehat{\tau}}\def\hrho{\widehat{\rho}}
\def\gbar{{\overline g}}\def\eps{\epsilon}\def\tbSigmas{{\hbSigma^s}}
\def\tbSigmas{\tbSigma^{s}}\def\hmatR{{\widehat\matR}}\def\hmatT{{\widehat\matT}}
\def\Sd{{\bbS^{d-1}}}\def\buj{{\bu_j}}\def\pa{\partial}
\newcommand{\bel}{\begin{eqnarray}\label}
\newcommand{\eel}{\end{eqnarray}}
\newcommand{\bes}{\begin{eqnarray*}}
\newcommand{\ees}{\end{eqnarray*}}
\newcommand{\norm}[2]{\|#1\|_{#2}}\renewcommand{\qed}{\hfill \mbox{\raggedright \rule{0.1in}{0.1in}}}\def\hbSigmat{\hbSigma^{\tau}}
\def\hbSigmar{\hbSigma^{\rho}}
\def\Var{\hbox{Var}}
\def\bU{\mathbold{U}}\def\bV{\mathbold{V}}

\def\bZ{\mathbold{Z}}

\newcommand{\Deltajkt}{\Delta^{\tau}_{jk}}\newcommand{\Deltajkr}{\Delta^{\rho}_{jk}}
\newcommand{\bDeltar}{\bDelta^{\rho}}\newcommand{\bDeltat}{\bDelta^{\tau}}

\newcommand{\Sigjk}{\Sigma_{jk}}

\newcommand{\hbthetaspt}{\hbtheta^{\tau}_{1;s}}
\newcommand{\hbthetaspr}{\hbtheta^{\rho}_{1;s}}

\newcommand{\hSigjkr}{\Sigmahat^{\rho}_{jk}}\def\hSigma{\widehat{\Sigma}}

\newcommand{\Unjk}{U_{n;jk}}\newcommand{\bDelzero}{\bDelta_{n}^{(0)}}
\newcommand{\bDelone}{\bDelta_{n}^{(1)}}
\def\Re{{\mathbb{R}}}
\def\bW{\mathbold{W}}

\newcommand{\hbarzero}[1]{\overline{h}_{0}\left({#1}\right)}

\def\hbSigmatt{\hbSigma_{(k)}^{\tau-taper}}
\def\hbSigmart{\hbSigma_{(k)}^{\rho-taper}}

\def\hmatR{\widehat{\matR}}
\def\hmatT{\widehat{\matT}}\def\buj{\bu_{(j)}}
\def\Sd{\bbS^{d-1}}\newcommand{\fu}[1]{f_{\bu}\left({#1}\right)}
\def\sgn{\hbox{\rm sgn}}
\def\scrA{{\mathscr A}}
\endlocaldefs

\begin{document}

\begin{frontmatter}
\title{MULTIVARIATE ANALYSIS OF NONPARAMETRIC\\ ESTIMATES OF LARGE CORRELATION MATRICES}
\runtitle{Nonparametric Correlation Matrix Convergence}

\begin{aug}
\author{\fnms{Ritwik} \snm{Mitra}\ead[label=e1]{ritwikm@eden.rutgers.edu}}
\address{Ritwik Mitra\\
Department of Statistics\\ and Biostatistics\\
Hill Center, Busch Campus\\
Rutgers University\\
Piscataway, New Jersey 08854\\
USA
\printead{e1}}
\and
\author{\fnms{Cun-Hui} \snm{Zhang}\ead[label=e2]{czhang@stat.rutgers.edu}}
\address{Cun-Hui Zhang\\
Department of Statistics\\ and Biostatistics\\
Hill Center, Busch Campus\\
Rutgers University\\
Piscataway, New Jersey 08854\\
USA
\printead{e2}}
\runauthor{R. Mitra \& C.-H. Zhang}
\affiliation{Rutgers University}

\end{aug}

\begin{abstract}
We study concentration in spectral norm of nonparametric estimates of correlation matrices. 
We work within the confine of a Gaussian copula model. 
Two nonparametric estimators of the correlation matrix, 
the sine transformations of the Kendall's tau and Spearman's rho correlation coefficient, are studied. 
Expected spectrum error bound is obtained for both the estimators. 
A general large deviation bound for the maximum spectral error of a collection of 
submatrices of a given dimension is also established. 
These results prove that when both the number of variables and sample size are large, 
the spectral error of the nonparametric estimators is of no greater order than that of the latent sample 
covariance matrix, at least when compared with some of the sharpest known error bounds for the later. 
As an application, we establish the minimax optimal convergence rate in the estimation of high-dimensional 
bandable correlation matrices via tapering off of these nonparametric estimators. 
An optimal convergence rate for sparse principal component analysis is also established 
as another example of possible applications of the main results.  
\end{abstract}

\begin{keyword}[class=MSC]
\kwd[Primary ]{62H12}
\kwd{62G05}
\kwd[; secondary ]{62G20}
\end{keyword}

\begin{keyword}
\kwd{Bandable}
\kwd{Correlation matrix}
\kwd{Gaussian copula}
\kwd{High dimension}
\kwd{Hoeffding decomposition}
\kwd{Kendall's tau}
\kwd{Nonparametric}
\kwd{Sparse PCA}
\kwd{Spearman's rho}
\kwd{Spectral norm}
\kwd{Tapering}
\kwd{U-statistics}
\end{keyword}

\end{frontmatter}

\section{Introduction}
We consider $n$ iid copies $\{\bX_{i}:1\leq i\leq n\}$ of a $d$-dimensional Gaussian random vector 
$(X_1,\ldots,X_d)^T$. 
We define $\matX = (\bX_{1},\cdots,\bX_{n})^{T} \in \Re^{n\times d}$. 
We assume that $\bX_{i}$'s are centered and marginally scaled, so that $\bbE\bX=\bzero$ and the correlation 
matrix is given by $\bbE\bX\bX^{T}/n=\bSigma\in \Re^{d\times d}$ with 1 in the diagonal. 
In this paper, we work within a high-dimensional `double asymptotic' setting where $d\wedge n\rightarrow \infty$. 
We assume that instead of $\bX$, we only observe $n$ iid copies $\bY_{i}, 1\leq i\leq n$, of the transformed variables
\begin{align*}
(f_{1}(X_{1}), \cdots, f_{d}(X_{d}))^{T}
\end{align*}
where $f_{i}$'s are unknown but strictly increasing. This is a form of the copula model \citep{Sklar59} for the 
distribution of the data. Because $\bX$ follows a Gaussian distribution, it is a formulation of the Gaussian copula, 
cf. \cite{BKRW93} and references therein. A slightly different but equivalent formulation of the Gaussian copula 
has been referred to as the \emph{nonparanormal} model \citep{LLW09}. 
Let $\matY = (\bY_{1},\cdots,\bY_{n})^{T}$. 
Our goal here is to estimate the latent correlation structure $\bSigma$ using the observed data matrix $\matY$. 

If we could observe the latent data matrix $\matX$, an obvious choice as an estimator would be the sample correlation matrix given by $\tbSigmas=\matX^{T}\matX/n$. 
It is for this reason that we refer to the latent $\tbSigmas$ as an oracle estimator. 
It is also clear that $\tbSigmas$ is a sufficient statistic for estimating $\bSigma$ when $\matX$ is known. As a consequence, any statistical procedure based on $\bSigma$ could be summarily described as  $g(\tbSigmas)$ for some function $g$. In this respect, $\tbSigmas$ possesses  great utility as an ideal raw estimate that lends itself to further analysis as the need be. 

However, as noted above, we do not observe $\matX$ but unknown strictly monotone 
transformations of columns of it, $\matY$. 
Thus the sample correlation matrix based on $\matY$, i.e. $\matY^{T}\matY/n$, 
is in general inconsistent in estimating the latent correlation structure $\bSigma$. 
Two candidate nonparametric estimators in such a scenario are considered in this paper: 
Kendall's tau developed in \cite{Kendall38} and Spearman's rank correlation coefficient, 
developed by Charles Spearman in 1904. 
These are two widely used nonparametric measures of association. 
Their properties in fixed dimension have been studied in \cite{Kendall38,Kendall48}, \cite{Kruskal58} and many others. 
More recently, in high-dimensional scenarios, correlation matrix estimators based on these measures 
have been taken up for study in \cite{LHYLW12} and \cite{XZ12} among others. 

For the rest of this paper, we call $\hbSigmat$ the correlation matrix estimator based on Kendall's tau 
and call $\hbSigmar$ the one based on Spearman's rho. 
It will be interesting to study whether for any statistical procedure, say $g(\tbSigmas)$, 
based on the raw estimate $\tbSigmas$,  it is possible to provide justification for the use of 
$g(\hbSigmat)$ or $g(\hbSigmar)$ as a viable replacement. 
It is however cumbersome to study each individual procedure separately. On the other hand, if $g$ is sufficiently smooth with respect to some matrix norm, it would suffice to study the accuracy of $\hbSigmat$ and $\hbSigmar$ 
as estimates of $\bSigma$ in such norms. 

A complete description of properties of $\hbSigmat$ and $\hbSigmar$ as estimators of large $\bSigma$ 
necessitates the derivation of the distributions of these matrix estimators. 
It is well known that in the multivariate Gaussian model, $\tbSigmas$ follows a Wishart distribution \citep{Anderson58}. 
To the contrary, derivation of the distribution of $\hbSigmat$ and $\tbSigmas$ seems at the present moment intractable. 
On the other hand, analysis of these nonparametric estimators for each individual element of the correlation matrix 
has been taken upon before. 
Both Kendall's tau and Spearman's rho are specific instances of U-statistics with bounded kernels. 
In \cite{Hoeff48}, the asymptotic normality of these nonparametric estimators for an individual correlation 
was established. 
Furthermore, the celebrated \cite{Hoeff63} inequality provides large deviation bounds for these 
estimators as U-statistics with bounded kernels. 
These results provide tools for studying the concentration of $\hbSigmat$ and $\hbSigmar$ 
in the matrix $\max$ norm and its applications \citep{LHYLW12,XZ12} and the corresponding 
Gaussian copula graphical model \citep{LHZ12}. 

It is important to note that while estimation accuracy in one specific matrix norm could be more appropriate for a certain set of statistical problems, some other set of problems might require accuracy in a different matrix norm. In this paper we focus on the spectral norm, which is also understood as the $\ell_{2}$ operator norm. 
Many statistical problems can be studied with error bounds in the spectral norm of estimated correlation matrices. 
A primary example is the principal component analysis (PCA) since the spectral norm is essential in studying the 
effects of matrix perturbation on eigenvalues and eigenvectors. 

Before beginning the study of convergence of $\hbSigmat$ and $\hbSigmar$ in the spectral norm, 
it is worthwhile to note that convergence rate of the latent sample covariance matrix $\tbSigmas$ 
in the spectral norm has been studied widely and established in a multitude of literature. 
A detailed overview and further references can be found in \cite{Ver10} among others. 
For example, one could derive, from the concentration inequality in Theorem II.13 of \cite{DSz01}, 
that for $\matX \in \Re^{n\times d}$ with iid $N(\bzero, \bSigma)$ rows, 
\begin{align}\label{eq:dsz}
\sqrt{\bbE\norm{\tbSigmas-\bSigma}{S}^{2}} \leq \norm{\bSigma}{S}\left(2\sqrt{2}\sqrt{d/n}
+\sqrt{2}d/n + 6(d/n^{3})^{1/4}\right), 
\end{align}
so that the consistency of $\tbSigmas$ follows when $d/n\to 0$. 
Additionally, the concentration inequality also provides a uniform bound on the spectral error for any 
$s$-dimensional diagonal submatrix for larger $d$. 
Taking any integer $s<d$ and sets $A\subset \{1,\cdots,d\}$, we have by the union bound 
\bel{eq:dsz1}
&& \max_{|A|\leq s}\norm{(\tbSigmas-\bSigma)_{A\times A}}{S}\Big/\max_{|A|\le s}\norm{\bSigma_{A\times A}}{S} 
\\ \nonumber &\le& \left(\sqrt{s/n}+\sqrt{2\big\{t+\log\hbox{${d\choose s}$}\big\}/n}\right)
\left(2+\sqrt{s/n}+\sqrt{2\big\{t+\log\hbox{${d\choose s}$}\big\}/n}\right)
\eel
with at least probability $1-2e^{-t}$. 
These spectral error bounds are explicit and of sharp order for the latent sample correlation matrix estimate 
$\tbSigmas$. In this light, it is apt to ask whether $\hbSigmat$ and $\hbSigmar$ also submit 
similar error bounds. 

In \cite{LH13a} a rate of $\sqrt{d\log d/n}$ was established for $\hbSigmat$ in a transelliptical 
family of distributions \citep{LHZ12}. 
In a separate but simultaneous work in \cite{WZ13} the same rate was established for $\hbSigmat$ 
in an elliptical copula correlation factor model, which can be also viewed as elliptical copula. 
In this paper, we provide non-asymptotic spectrum error bounds in the more restrictive Gaussian copula model  for both $\hbSigmat$ and $\hbSigmar$ which improve the convergence rates of these existing error bounds. 
In particular, we establish in Theorem \ref{th:GenUconc} expected spectral error bounds to match (\ref{eq:dsz}), 
and under mild conditions on the sample size, we establish in Theorem \ref{th:GenUconc2} and its corollaries 
large deviation bounds to match (\ref{eq:dsz1}). 
These results establish that in the Gaussian copula model the nonparametric estimators $\hbSigmat$ and $\hbSigmar$ 
perform as well as the oracle raw estimator $\tbSigmas$ in terms of the order of the spectral error. 
Consequently, a methodology based on $\tbSigmas$ that hinges on a spectrum error bound 
can be performed with the same rate of convergence if $\hbSigmat$ or $\hbSigmar$ are used in lieu 
of the latent $\tbSigmas$.

We discuss two different statistical problems where our results could be applied. 
The first, a ripe problem for application of spectral error bounds, 
is the estimation of a large bandable correlation matrix. 
For high-dimensional data, proper estimation of large bandable $\bSigma$ involves 
implementation of various regularization strategies such as banding, tapering, thresholding etc. 
These procedures and their properties have been studied in \cite{WuP03}, \cite{BL08a,BL08b}, 
\cite{Karoui08}, \cite{LamFan09}, \cite{cailiu11a}, \cite{CaiZhou12}, and \cite{CY12}.  
In particular, \cite{CZZ10} established the optimal minimax rate of convergence for a tapered version of 
$\tbSigmas$ for certain classes of unknown bandable $\bSigma$. 
In \cite{XZ12a}, a tapering estimator based on the Spearman's rank correlation was studied 
for the same class of parameters in the Gaussian copula model. 
However, the question of whether the nonparametric estimator
could attain the optimal rate, was not resolved in their paper.
Our spectral error bounds imply that the optimal rate is attained 
if one substitutes $\tbSigmas$ with either $\hbSigmat$ or $\hbSigmar$.

The second application involves error bounds in the estimation of the leading eigenvector in PCA 
both with and without a sparsity assumption on the eigenvector. 
With the advent and increasing prevalence of high dimensional data, 
various limitations of traditional procedures had come to the fore. 
For instance, \cite{JL09} showed that when $d/n \rightarrow c >0$, the principal component of $\tbSigmas$ 
is inconsistent in estimating the leading eigenvector of the true correlation matrix. 
Several remedies to this problem have been proposed, all being different formulations under the auspice of a 
general sparse PCA paradigm. 
In sparse PCA, the eigenvectors corresponding to the largest eigenvalues are assumed to be sparse. 
A vast array of sparse PCA approaches has been proposed and studied in 
\cite{JTU03}, \cite{ZHT06}, \cite{DALJL07}, \cite{VL12}, \cite{Ma13}, and \cite{CMW13a} among others. 
For the elliptical copula family, \cite{LH13a} established the optimal rate of convergence 
in sparse PCA with $\hbSigmat$ under an additional sign sub-Gaussian condition. 
We will demonstrate that our spectral error bounds for the nonparametric estimators can be 
directly applied to study the convergence rates for the principle component direction. 
In particular, for sparse PCA the minimax rate as described in \cite{VL12} will be established 
without imposing the sign sub-Gaussian condition. 

Our work is organized as follows. In Section \ref{sec:Background} we describe the Gaussian copula model 
and the Kendall's tau and Spearman's rho estimators for the correlation matrix. 
In Section \ref{sec:Especnorm}, we provide upper bounds for the expected spectral error 
for these two correlation-matrix estimators in Theorem \ref{th:GenUconc} and 
outline our analytical strategy. 
In Section \ref{sec:LargeDev}, we provides a general large deviation inequality in Theorem \ref{th:GenUconc2}. 
In Section \ref{sec:Application} we discuss two problems where our results on spectral norm concentration could be utilized. Some of the proofs are relegated to the Appendix.
\section{Background \& Preliminary Results}\label{sec:Background}
We describe the basic data model and define the nonparametric estimates of $\bSigma$. 

\subsection{Data Model and Notation}\label{subsec:datmodel}
 We consider the Gaussian copula or multivariate nonparametric transformational model 
\begin{equation}
(Y_{1},\cdots,Y_{d})^{T} = (f_{1}(X_{1}),\cdots, f_{d}(X_{d}))^{T}, 
 \label{eq:mod1}
\end{equation}
where $(X_{1},\cdots,X_{d})^{T} \in \Re^{d}$ is a multivariate Gaussian random vector with 
marginal $N(0,1)$ distribution and $f_{j}$ are unknown strictly increasing functions. 
We are interested in estimating the population correlation matrix of $(X_{1},\cdots,X_{d})^{T}$, denoted by 
\begin{equation}
 \bSigma=\bbE(X_{1},\cdots,X_{d})^{T}(X_{1},\cdots,X_{d}), 
 \label{eq:corr}
\end{equation}
based on a sample of iid copies of $(Y_{1},\cdots,Y_{d})^{T}$. 
Since the $f_{j}$ absorbs the location and scale of the individual $X_{j}$, 
it is natural to assume $\bbE X_{j}=0$ and $\bbE X_{j}^2=1$ on the marginal distribution. 

The observations $\bY_{i}= (Y_{i1},\cdots,Y_{id})^{T}$, $i=1,\cdots,n$, are iid copies of 
$(Y_{1},\cdots,Y_{d})^{T}$. They can be written as 
\begin{equation}
 Y_{ij} = f_{j}(X_{ij}) \quad i=1\cdots,n\quad j=1,\cdots, d, 
 \label{eq:mod2}
\end{equation}
where $\bX_{i}=(X_{i1},\cdots,X_{id})^{T}\in \Re^{d}$ are independent copies of 
$(X_{1},\cdots,X_{d})^{T}\sim N(\bzero,\bSigma)$ in (\ref{eq:mod1}). 
We denote by $\matX = (\bX_{1},\cdots,\bX_{n})^{T} \in \Re^{n\times d}$ the matrix with rows 
$\bX_{i}^T$ and quite similarly $\matY = (\bY_{1},\cdots,\bY_{n})^T \in \Re^{n\times d}$. 

We use the following notation throughout the paper. 
For vectors $\bu \in \Re^{d}$, the $\ell_{p}$ norm is denoted by 
$\norm{\bu}{p} = \left(\sum^{d}_{k=1}|u_{k}|^{p}\right)^{1/p}$,  
with $\|\bu\|_\infty = \max_{1\le k\le d}|u_k|$ and $\norm{\bu}{0}= \#\{j:u_{j}\neq 0\}$. 
For matrices $\matA=(A_{jk})_{d\times d}\in \Re^{d\times d}$, the $\ell_{p}\rightarrow \ell_q$ operator norm is denoted by 
$\norm{\matA}{(p,q)} = \max_{\norm{\bu}{p}=1}\norm{\matA\bu}{q}$. 
The $\ell_2 \rightarrow \ell_2$ operator norm, known as the spectrum norm, is 
\[\norm{\matA}{S}=\norm{\matA}{(2,2)} 
= \max_{\norm{\bu}{2}=1}{|\bu^{T}\matA\bu|}\] 
The vectorized $\ell_\infty$ and Frobenius norms are denoted by 
 \begin{align*}
 \norm{\matA}{\max} = \max_{j,k}|A_{jk}|,\quad \norm{\matA}{F}=\sqrt{\trace\left(\matA^{T}\matA\right)}. 
 \end{align*}
 For symmetric matrices $\matA$, the $j^{th}$ eigenpair of $\matA$ is denoted by $\lambda_{j}(\matA)$ 
 and $\btheta_{j}(\matA)$, so that $\lambda_{1}(\matA)=\norm{\matA}{S}$ and $\btheta_{1}(\matA)$ is the leading eigenvector. 
In addition to $\bbE$ and $\bbP$, which denote the expectation and probability measure, 
we denote by $\bbE_n$ the average over iid copies of variables in (\ref{eq:mod2}). For example, 
\bes
\bbE_nh(x_j,x_k) = n^{-1}\sum_{i=1}^nh(X_{ij},X_{ik}). 
\ees
The relation $a_{n}=\calO(b_{n})$ will imply $a_{n}\leq Kb_{n}$ for some fixed constant $K>0$. Finally we denote $\bbS^{d-1} = \{\bu\in \bbR^d:\ \norm{\bu}{2}=1\}$. 
\subsection{Nonparametric Estimation of Correlation Matrix}\label{subsec:nonparam} 
The approach we adopt in estimating the correlation matrix $\bSigma =(\Sigma_{jk})$ in (\ref{eq:corr}) is 
based on Kendall's tau ($\tau$) or Spearman's correlation coefficient rho ($\rho$). 

With the observations $Y_{ij}$ in (\ref{eq:mod2}), Kendall's tau is defined as 
\begin{equation}
\htau_{jk} = 
              \dfrac{2}{n(n-1)}\sum_{1\leq i_1<i_2\leq n}\sgn(Y_{i_1j}-Y_{i_2j})\sgn(Y_{i_1k}-Y_{i_2k}),
              \label{eq:tauy}
\end{equation}
and Spearman's rho as
\begin{equation}
 \hrho_{jk} = 
               \dfrac{\sum_{i=1}^{n}(r_{ij}-(n+1)/{2})(r_{ik}-(n+1)/{2})}
               {\sqrt{\sum_{i=1}^{n}(r_{ij}-(n+1)/{2})^{2}\sum_{i=1}^{n}(r_{ik}-(n+1)/{2})^{2}}} , 
               \label{eq:rhoy}
\end{equation}
where $r_{ij}$ is the rank of $Y_{ij}$ among $Y_{1j},\cdots, Y_{nj}$. In matrix notation, 
\bel{matrix-est}
\hmatT=(\htau_{jk})_{d\times d},\quad \hmatR=(\hrho_{jk})_{d\times d}. 
\eel

The population version of Kendall's tau is given by
\bel{pop-tau}
 \tau_{jk} = \bbE\,\sgn(Y_{1j}-Y_{2j})\sgn(Y_{1k}-Y_{2k}), 
\eel
while the population version of Spearman's rho is given by
\bel{pop-rho}
 \rho_{jk} = 3\,\bbE\,\sgn(Y_{1j}-Y_{2j})\sgn(Y_{1k}-Y_{3k}). 
\eel
In matrix notation, the population version of (\ref{matrix-est}) is 
\bel{matrix-pop}
\matT=(\tau_{jk})_{d\times d},\quad \matR=(\rho_{jk})_{d\times d}.
\eel

Since $f_{j}$ are strictly increasing functions, we have $\sgn(f_{j}(u)-f_{j}(v)) = \sgn(u-v)$. Thus, Kendall's tau, Spearman's rho and their population version are unchanged if the observed $\bY = (Y_{ij})_{n\times d}$ is replaced by the unobserved $\bX = (X_{ij})_{n\times d}$ in their definition. 
Since $X_{j}$  follows a standard normal distribution, we have, from \cite{Kendall48} and \cite{Kruskal58}, that for $\Sigma_{jk}=\bbE X_{j}X_{k}$, 
\begin{equation}\label{pop-eq}
 \Sigma_{jk} = \sin\left(\dfrac{\pi}{2}\tau_{jk}\right)  = 2\sin\left(\dfrac{\pi}{6}\rho_{jk}\right). 
\end{equation}
This immediately leads to the following correlation matrix estimator by Kendall's tau, 
\begin{equation}
\hbSigma^{\tau} = (\Sigmahat^{\tau}_{jk})_{d\times d},\qquad 
\Sigmahat^{\tau}_{jk} = 
                          \sin\left(\dfrac{\pi}{2}\htau_{jk}\right). 
                          \label{eq:sigesttau}
\end{equation}
In the same light we define the correlation matrix estimator by Spearman's rho as 
\begin{equation}
\hbSigma^{\rho} = (\Sigmahat^{\rho}_{jk})_{d\times d},\qquad 
\Sigmahat^{\rho}_{jk} = 2\sin\left(\dfrac{\pi}{6}\hrho_{jk}\right). 
  \label{eq:sigestrho}
\end{equation}

The following proposition states a slightly different version of Theorem 2.3 of \cite{WZ13} 
and a direct application of their argument to Spearman's rho. 

\begin{proposition}\label{prop-norms} 
Both matrices $\matT-(2/\pi)\bSigma$ and $\matR-(3/\pi)\bSigma$ are nonnegative-definite, 
$\|\matT-(2/\pi)\bSigma\|_S \le (1-2/\pi)\|\bSigma\|_S$, 
and $\|\matR-(3/\pi)\bSigma\|_S \le (1-3/\pi)\|\bSigma\|_S$. Consequently, 
\bel{prop1-1}
\|\matT\|_S\vee\|\matR\|_S \le \|\bSigma\|_S. 
\eel
\end{proposition}


\section{Expected Spectrum Error Bounds}\label{sec:Especnorm}
While Spearman's rho and Kendall's tau are structurally different, they can be represented neatly as 
U-statistics of a special type. 
In this section we develop bounds for the expected spectrum norm 
of their error via a certain decomposition of such U-statistics. 
This decomposition also provides an outline of our analysis of the concentration of the spectrum 
norm and the sparse spectrum norm of the error in subsequent sections. 


Given a sequence of $n$ observations from a population in $\Re^d$, a matrix U-statistic with order $m$ and 
kernels $h_{jk}(\bx_1,\ldots,\bx_m)$ can be written as 
\begin{equation}
 \matU_{n} = \left(\Unjk\right)_{d\times d}
 \label{eq:Umat}
\end{equation}
with elements 
\begin{equation}
 \Unjk = \dfrac{(n-m)!}{n!} \sum_{1\leq i_{1}\neq \cdots\neq i_{m}\leq n} 
 h_{jk}(\bX_{i_{1}},\bX_{i_{2}},\cdots,\bX_{i_{m}}). 
\end{equation}

Assume that $h_{jk}(\bx_1,\ldots,\bx_m)$ are permutation symmetric and set 
\begin{equation}
\quad \hbar_{jk}(\bx) = \bbE \left[h_{jk}(\bX_{{1}},\cdots,\bX_{{m}})\Big|\bX_{{1}} = \bx\right] -c_{jk}
  \label{eq:hjk}	
\end{equation}
with any constants $c_{jk}$.
The Hoeffding decomposition of $\matU_{n}$ can be written as 
\begin{align}
 \matU_{n} - \bbE \matU_{n} = \sum_{\ell=1}^m \binom{m}{\ell}\bDelta_{n}^{(\ell)} 
 \label{eq:hoefgen}
\end{align}
where $\bDelone$ is an average of iid random matrices with elements 
\begin{equation}
 \Delta_{n;jk}^{(1)} = (\bbEn -\bbE)\hbar_{jk} 
 = \frac{1}{n}\sum_{i=1}^n\Big(\hbar_{jk}(\bX_{i}) -\bbE \hbar_{jk}(\bX_1)\Big) 
\end{equation}
and $\bDelta_n^{(\ell)} = (\Delta_{n;jk}^{(\ell)})_{d\times d}$ are matrix U-statistics with completely 
degenerate kernels of order $\ell$.  We refer to \cite{Hoeff48}, \cite{HSS67}, \cite{Hajek68}, \cite{Van00} and \cite{Serfling09} for 
 detailed exposition on the Hoeffding decomposition and additional references. 

Since the components of the Hoeffding decomposition are orthogonal, 
\begin{align*}
\bbE \left(\sum_{\ell=2}^m \binom{m}{\ell}\Delta_{n;jk}^{(\ell)}\right)^2 
 & = \sum_{\ell=2}^m{n\choose \ell}^{-1}{m\choose \ell}^3 \bbE \left(\Delta_{m;jk}^{(\ell)}\right)^2  \\
 & \leq \binom{n}{2}^{-1}\binom{m}{2}\Var\big(h_{jk}\big(\bX_{{1}},\cdots,\bX_{{m}})\big). 
\end{align*}
A consequence of the above calculation of variance is 
\bes
&&\bbE \left\|\matU_{n} - \bbE \matU_{n}-m\bDelone\right\|_F^2 
\le \frac{m(m-1)}{n(n-1)}\sum_{j=1}^d\sum_{k=1}^d\Var\big(h_{jk}\big(\bX_{{1}},\cdots,\bX_{{m}})\big). 
\ees
We note that Kendall's tau and Spearman's rho are U-statistics of order $m=2$ and $3$ respectively, 
both with kernels satisfying $h_{jj}(\bx_{{1}},\cdots,\bx_{{m}})=1$ and 
$\|h_{jk}(\bx_{{1}},\cdots,\bx_{{m}})\|_\infty \le 1$ for $j\neq k$. 
It follows that the high order terms of their Hoeffding decompositions 
are explicitly bounded by  
\bel{var-Delta-2}
\bbE \left\|\matU_{n} - \bbE \matU_{n}-m\bDelone\right\|_F^2 
\le \frac{m(m-1)d(d-1)}{n(n-1)}. 
\eel

Now we consider the term $\bDelone$. 
It turns out that in the Gaussian copula model (\ref{eq:mod2}), 
the first order kernel for Kendall's tau can be written as
\bes
\hbar_{jk}(x_1,\ldots,x_d) = \begin{cases} \hbar(x_j,x_k,\Sigma_{jk}),& j\neq k \cr 1 & j=k \end{cases}
\ees
with $\hbar(x_j,x_k,0) 
=\hbar_0(x_j)\hbar_0(x_k)$, 
where $\hbar_0(x) = 2\Phi(x)-1$, and that of Spearman's rho is of the same form. 
This motivates a further decomposition of $\bDelone$ 
as a sum of $\bDelzero$ and $\bDelone - \bDelzero$, with 
\bel{Delta-zero}
&& \bDelzero = \Big(\Delta_{n;jk}^{(0)}\Big)_{d\times d}
= \Big((\bbEn-\bbE)\hbar_0(x_j)\hbar_0(x_k)\Big)_{d\times d}, 
\\ \nonumber && \bDelone - \bDelzero 
= \Big((\bbE_n-\bbE)\big(\hbar(x_j,x_k,\Sigma_{jk})-\hbar(x_j,x_k,0)\big)\Big)_{d\times d} 
\eel
It follows from the definition of the population Spearman's rho in 
(\ref{pop-rho}) that 
\bes
\bbE\,\hbar(X_j,X_k,0) = \bbE\,\hbar_0(X_j)\hbar_0(X_k) = \rho_{jk}/3,\quad \forall 1\le j\le k\le d. 
\ees 
Thus, the $\bDelzero$ in (\ref{Delta-zero}) can be written as  the difference between the sample covariance 
matrix of $\hbar_0(\bX)=(\hbar_0(X_{ij}))_{n\times d}$ and its expectation: 
\bel{Delta-zero2}
\bDelzero = n^{-1}\hbar_0(\bX)^{T}\hbar_0(\bX) - \matR/3. 
\eel
Moreover, we will prove that for both Kendall's tau and Spearman's rho
\bel{g-bound}
\Big| \hbar(x_j,x_k,\Sigma_{jk}) - \hbar(x_j,x_k,0)\Big| \le C_1\Big|\Sigma_{jk}\Big|,\ j\neq k. 
\eel
with $C_1= 2/\pi+1\le 2$ for Kendall's tau and $C_1\le 1+\sqrt{8}/\pi\le 2$ for Spearman's rho.   
Thus, since $\Var(\hbar_0^2(X_{ij}))=\int_0^1((2x-1)^2-1/3)^2dx = 4/45$ on the diagonal of $\bDelone - \bDelzero$ 
and $\bDelone - \bDelzero$ is an average of iid matrices, 
\bel{Delta-01-bd}
&& \bbE \Big\| \bDelone - \bDelzero\Big\|_S^2 
\le \bbE \Big\| \bDelone - \bDelzero\Big\|_F^2 
\le C_1^2\sum_{j\neq k}\frac{\Sigma_{jk}^2}{n}+\frac{4d}{45 n}. 
\eel

Let $\matU_{n}$ be the matrix U-statistics of either Kendall's tau or Spearman's rho,  
$\matU_{n}=\hmatT=(\htau_{jk})_{d\times d}$ or $\matU_{n}=\hmatR=(\hrho_{jk})_{d\times d}$ 
as in (\ref{matrix-est}) respectively, and 
$\hbSigma$ the corresponding estimator of $\bSigma$ in (\ref{eq:sigesttau}) and (\ref{eq:sigestrho}). 
It follows from the expansion of the sine function in 
(\ref{eq:sigesttau}) and (\ref{eq:sigestrho}) that 
\bel{outline-1}
(\hbSigma - \bSigma)_{jk}\approx a_0(\matU_{n} - \bbE \matU_{n})_{jk}, 
\eel
with $a_0=\pi/2$ for $\bU_n=\hmatT$ and $a_0=\pi/3$ for $\bU_n=\hmatR$. 
Thus, the estimators $\hbSigma$ can be decomposed as 
\bel{outline-2}
\hbSigma - \bSigma
&=& a_0\Big\{(\matU_{n} - \bbE \matU_{n}) - m\bDelone\Big\} +  a_0m\Big(\bDelone-\bDelzero\Big)
\cr && + a_0m\bDelzero + \Big\{(\hbSigma - \bSigma) - a_0(\matU_{n} - \bbE \matU_{n})\Big\}, 
\eel
where the first two terms are bounded by (\ref{var-Delta-2}) and (\ref{Delta-01-bd}) respectively 
and the third term is explicitly expressed as the difference between a sample covariance 
matrix and its expectation in (\ref{Delta-zero}). 
Moreover, the fourth term can be bounded with a higher order expansion of $\sin(t)$ in 
(\ref{eq:sigesttau}) and (\ref{eq:sigestrho}). 
We note that the fourth term on the right-hand side of (\ref{outline-2}) is not needed if one is 
interested in studying $\hmatT-\matT$ or $\hmatR-\matR$ without the sine transformation. 
This analysis leads to the following theorem. 

\begin{theorem}\label{th:GenUconc}   
Let $\hmatT$ and $\hmatR$ be respectively the Kendall's tau and Spearman's rho matrices in (\ref{matrix-est}), 
$\matT$ and $\matR$ be their population version in (\ref{matrix-pop}), and 
$\hbSigma^\tau = (\Sigmahat^{\tau}_{jk})_{d\times d}$ and 
$\hbSigma^\rho= (\Sigmahat^{\rho}_{jk})_{d\times d}$ be the corresponding estimators 
in (\ref{eq:sigesttau}) and (\ref{eq:sigestrho}) for the population 
correlation matrix $\bSigma$ in the Gaussian copula model (\ref{eq:mod1}). 
Then, for certain numerical constant $C_0$ and both $\hbSigma = \hbSigma^\tau$ and $\hbSigma = \hbSigma^\rho$\bel{th-1-1} 
&& \bbE\norm{\hbSigma - \bSigma}{S} 
+ \bbE \norm{\hmatT- \matT}{S} + \bbE \norm{\hmatR- \matR}{S}
\le C_0\|\bSigma\|_S\Big(\sqrt{d/n} + d/n\Big).
\eel
In particular, defining $n_{2}=2\lfloor n/2\rfloor$ (where $\lfloor x\rfloor$ is the integer part of $x$),
\bel{th1-2} 
\bbE \norm{\hmatT- \matT}{S}
&\le & \sqrt{2d(d-2n)_+/\{n(n-1)\}+4(2/\pi+1)^2\norm{\bSigma}{F}^2/n}
\cr && + 10\|\bSigma\|_S\Big(\sqrt{(d+1)/(3n)} + (d+1)/n\Big), 
\\ \nonumber \bbE \norm{\hbSigma^\tau - \bSigma}{S}
&\le& \frac{\pi}{2}\bbE \norm{\hmatT- \matT}{S}
+\frac{\pi}{2}\sqrt{\frac{\|\bSigma\|_F^2-d}{n_2}}+\frac{\pi^2\sqrt{3}d}{8n_2}, 
\eel
for Kendall's tau, and for Spearman's rho, with $n_{3}=3\lfloor n/3\rfloor$
\bel{th1-3}
\bbE \norm{\hmatR- \matR}{S}
&\le & \sqrt{6d(d-2n)/\{n(n-1)\}+9(1+\sqrt{8}/\pi)^2\norm{\bSigma}{F}^2/n}
\cr && + 15\|\bSigma\|_S\Big(\sqrt{(d+1)/(3n)} + (d+1)/n\Big)+\|\bSigma\|_F/n,
\\ \nonumber \bbE \norm{\hbSigma^\rho - \bSigma}{S}
&\le& \frac{\pi}{3}\bbE \norm{\hmatR- \matR}{S}
+\frac{\pi}{9}\sqrt{\frac{\|\bSigma\|_F^2-d}{n_3}}+\frac{\pi^2\sqrt{3}d}{36 n_3} + \frac{2\pi\|\bSigma\|_F}{3n}.
\eel 
\end{theorem}

\begin{corollary}\label{cor:GenUconc}
If $\|\bSigma\|_S^2d/n\to 0$, then
\bes
\bbE \norm{\hmatT- \matT}{S} + \bbE\norm{\hbSigma^\tau - \bSigma}{S} 
+ \bbE \norm{\hmatR- \matR}{S}+\bbE \norm{\hbSigma^\rho - \bSigma}{S}\to 0. 
\ees
\end{corollary}

\begin{remark}
Up to a numerical constant factor, Theorem \ref{th:GenUconc} match the bound (\ref{eq:dsz}) 
for the expected spectral error of the oracle sample covariance matrix $\tbSigmas$. 
While \cite{LH13a} and \cite{WZ13} focused on large deviation bound of the spectral error 
of $\norm{\hbSigmat - \bSigma}{S}$ in the elliptical copula model, a direct application of their results  
requires $\|\bSigma\|_S d(\log d)/n\to 0$ for the convergence in spectrum norm. 
Although their results are of sharper order when $\|\bSigma\|_S\gg \log d$, 
it seems that when $\|\bSigma\|_S=\calO(1)$, the extra logarithmic factor cannot be removed in their analysis 
based on the matrix Bernstein inequality \citep{Tropp11b}. 
\end{remark}

The proof of Theorem \ref{th:GenUconc} requires a number of inequalities which provide key details of the analysis outlined above the statement of the theorem. These inequalities are crucial for our derivation of large deviation spectrum error bounds as well. We state these inequalities in a sequence of lemmas below and defer their proofs to the Appendix.

Let $\varphi_\rho(x,y)$ be the bivariate normal density with mean zero, 
variance one, and correlation $\rho$. Define 
\begin{align}\label{hbar}
 \hbar(x,y,\rho) = \int\int \sgn(x-u)\sgn(y-v)\varphi_\rho(u,v)dudv. 
\end{align}

\begin{lemma}\label{lem:kernel} Let $\hbar(x,y,\rho)$ be as in (\ref{hbar}). 
Based on $\bX\in\Re^{n\times d}$ with iid $N(0,\bSigma)$ rows, 
Kendall's $\htau_{jk}$ is a U-statistic of order 2 with a permutation symmetric 
kernel $h_{j,k}(\bx_1,\bx_2)$ satisfying $|h_{jk}(\bx_1,\bx_2)| = 1$ and 
\bel{lm-kernel-1}
&&\bbE\Big[h_{jk}(\bX_1,\bX_2)\Big|\bX_1=\bx\Big] = \hbar(x_j,x_k,\Sigma_{jk})\ \forall\ j\neq k.
\eel
With $g(x,y,\rho) = \hbar(x,y,\rho) - \hbar(x,y,0)$ and  $C_1=2/\pi+1$, 
\bel{lm-kernel-2}
&& \big|g(x,y,\rho)\big| \le C_1|\rho|,\ \big|(\pa/\pa x)g(x,y,\rho)\big| \le |\rho|. 
\eel
Moreover, with $\hbar_0(x)=2\Phi(x)-1$ and $\rho_{jk}$ in (\ref{pop-rho}),
\bel{lm-kernel-3}
&& \hbar(x,y,0)=\hbar_0(x)\hbar_0(y),\ \bbE\,\hbar(X_{ij},X_{ik},0) = \rho_{jk}/3\ \forall\ j,k.
\eel
\end{lemma}

\begin{lemma}\label{lm-kernel-r} Let $\hbar(x,y,\rho)$ be as in (\ref{hbar}) and $C_1=\sqrt{8}/\pi+1$. 
Based on $\bX\in\Re^{n\times d}$ with iid $N(\bzero,\bSigma)$ rows, 
Spearman's $\hrho_{jk}$ is a U-statistic of order 3 
with a permutation symmetric kernel $h_{j,k}^\rho(\bx_1,\bx_2,\bx_3)$ satisfying
\bel{lm-kernel-4}
&&|\bbE\hrho_{jk} - \rho_{jk}| \le |\rho_{jk}|/(n+1)\le |\Sigma_{jk}|/(n+1),
\\ \label{lm-kernel-5} &&  |h_{jk}^\rho(\bx_1,\bx_2,\bx_3)|\le 1,
\\ \label{lm-kernel-6} && \big|\hbar^\rho(x,y,\rho) - \hbar^\rho(x,y,0)\big| \le C_1|\rho|, 
\\ \label{lm-kernel-7} && \big|(\pa/\pa x)\big\{\hbar^\rho(x,y,\rho) - \hbar^\rho(x,y,0)\big\}\big| \le |\rho|, 
\\ \label{lm-kernel-8} && (1+1/n)\hbar^\rho(x,y,0) = \hbar(x,y,0), 
\eel
where $\hbar^\rho(x_j,x_k,\Sigma_{jk})=\bbE\big[h_{jk}^\rho(\bX_1,\bX_2,\bX_3)\big|\bX_1=\bx\big]-\tau_{jk}/(n+1)$. 
\end{lemma}

\begin{lemma}\label{lm-exp-sum}
Inequalities (\ref{var-Delta-2}) and (\ref{Delta-01-bd}) hold 
with $C_1= 2/\pi+1\le 2$ and $m=2$ for Kendall's tau 
and $C_1\le 1+\sqrt{8}/\pi\le 2$ and $m=3$ for Spearman's rho. 
Moreover, for both Kendall's tau and Spearman's rho,
\bel{lm-exp-sum-1}
&& \bbE \norm{(\matU_{n} - \bbE \matU_{n}) - m\bDelzero}{F}^2
\le \frac{m(m-1)d(d-2n)_+}{n(n-1)}+C_1^2\frac{\norm{\bSigma}{F}^2}{n/m^2}. 
\eel
\end{lemma}

\begin{lemma}  \label{lm-exp-zero}
Let $\bDelzero$ as in (\ref{Delta-zero}) and $\matR = (\rho_{jk})_{d\times d}$. Then, 
\bel{lm-exp-zero-1}
&& \bbE\|\bDelzero\|_{S} 
\le 5\|\bSigma\|_S\Big(\sqrt{(d+1)/(3n)}+(d+1)/n\Big)
\eel
and with at least probability $1-2e^{-t^2}$, 
\bel{lm-exp-zero-2}
\|\bDelzero\|_{S} &\le & 5\|\bSigma\|_S\Big(\sqrt{(d+t^2/\pi)/(3n)} + (d+(t^2+1)/\pi)/n\Big). 
\eel
\end{lemma}

\begin{lemma} \label{lm5} (i) Let $\hbSigmat =(\Sigmahat^\tau_{jk})_{d\times d}$ 
be as in (\ref{eq:sigesttau}) and $\bDeltat = (\Deltajkt)_{d\times d}$ with $\Deltajkt = \htau_{jk}-\tau_{jk}$. 
Let $n_2=2\lfloor n/2\rfloor$ where $\lfloor x \rfloor$ is the integer part of $x$. Then, 
\bel{lm5-1} 
&& \sqrt{\bbE\norm{(\hbSigmat - \bSigma) - (\pi/2)\bDeltat}{F}^2}
\le \frac{\pi}{2}\sqrt{\frac{\|\bSigma\|_F^2-d}{n_2}}+\frac{\pi^2\sqrt{3}d}{8n_2}. 
\eel
(ii) Let $\hbSigmar =(\Sigmahat^\rho_{jk})_{d\times d}$ 
be as in (\ref{eq:sigestrho}) and $\bDeltar = (\Deltajkr)_{d\times d}$ with $\Deltajkr = \hrho_{jk}-\bbE\hrho_{jk}$. 
Let $n_3=3\lfloor n/3\rfloor$ where $\lfloor x \rfloor$ is the integer part of $x$. 
Then, 
\bel{lm5-2} 
&& \sqrt{\bbE\norm{(\hbSigmar - \bSigma) - (\pi/2)\bDeltar}{F}^2}
\le \frac{\pi}{9}\sqrt{\frac{\|\bSigma\|_F^2-d}{n_3}}+\frac{\pi^2\sqrt{3}d}{36 n_3} 
+ \frac{\pi \sqrt{\|\bSigma\|_F^2-d}}{3(n+1)}. 
\eel
\end{lemma}

\begin{proof}[Proof of Theorem \ref{th:GenUconc}] 
Let $n_m=m\lfloor n/m\rfloor$. As in (\ref{outline-2}), for Kendall's tau, 
\bes
\norm{\hmatT- \matT}{S}
&\le & \Big\|(\matU_{n} - \bbE \matU_{n}) - 2\bDelzero\Big\|_{F} + 2\Big\|\bDelzero\Big\|_{S}, 
\cr \norm{\hbSigma^\tau - \bSigma}{S}
&\le & \Big\|(\hbSigma^\tau - \bSigma) - (\pi/2)(\matU_{n} - \bbE \matU_{n})\Big\|_{F} + (\pi/2)\Big\|\hmatT- \matT\Big\|_{S}. 
\ees
with $\matU_{n}=\hmatT$ and $\bbE \matU_{n}=\matT$. 
It follows from (\ref{lm-exp-sum-1}) of Lemma \ref{lm-exp-sum} with $m=2$, 
(\ref{lm-exp-zero-1}) of Lemma \ref{lm-exp-zero} and (\ref{lm5-1}) of Lemma \ref{lm5} that 
the inequalities in (\ref{th1-2}) hold. 

Similarly, for Spearman's rho, 
\bes
\norm{\hmatR- \matR}{S}
&\le & \Big\|(\matU_{n} - \bbE \matU_{n}) - 3\bDelzero\Big\|_{F} + 3\Big\|\bDelzero\Big\|_{S}
+\Big\| \bbE \matU_{n}-\matR\Big\|_{F}, 
\cr \norm{\hbSigma^\rho - \bSigma}{S} 
&\le & \Big\|(\hbSigma^\rho - \bSigma) - (\pi/3)(\hmatR- \matR)\Big\|_{F} 
+ (\pi/3)\norm{\hmatR- \matR}{S},
\ees
with $\matU_{n}=\hmatR$ and $\|\bbE \matU_{n}-\matR\|_F=\|(\bbE\hrho_{jk}-\rho_{jk})_{d\times d}\|_F\le \sqrt{\|\bSigma\|_F^2-d}/(n+1)$ 
by (\ref{lm-kernel-4}). 
Thus, (\ref{lm-exp-sum-1}), (\ref{lm-exp-zero-1}) and (\ref{lm5-2}) yield the inequalities in (\ref{th1-3}). 
\end{proof}

\section{Large Deviation Inequalities}\label{sec:LargeDev} 
While the upper bounds for the expected spectral error in Theorem \ref{th:GenUconc} and Corollary \ref{cor:GenUconc} 
match (\ref{eq:dsz}) for the oracle sample covariance matrix, it is useful only when $d/n\to 0$ as is the case in many 
applications. For $d > n$, large deviation bounds for the sparse spectral norm of the form (\ref{eq:dsz1}) 
is often used instead. 
In the present section we provide large deviation inequalities for both the spectral norm and the sparse spectral norm 
of the error for Kendall's tau and Spearman's rho. 

The main result for this section is a large deviation bound in the following theorem for the 
maximum spectral error in a collection of diagonal submatrices. 

\begin{theorem}\label{th:GenUconc2} 
Let $\hmatT$ and $\hmatR$ be respectively the Kendall's tau and Spearman's rho matrices in (\ref{matrix-est}), 
$\matT$ and $\matR$ be their population version in (\ref{matrix-pop}), and 
$\hbSigma^\tau = (\Sigmahat^{\tau}_{jk})_{d\times d}$ and 
$\hbSigma^\rho= (\Sigmahat^{\rho}_{jk})_{d\times d}$ be the corresponding estimators 
in (\ref{eq:sigesttau}) and (\ref{eq:sigestrho}) for the population 
correlation matrix $\bSigma$ in the Gaussian copula model (\ref{eq:mod1}). 
Let $1\le s\le d$, $m\ge 1$ and $\scrA_{s,m}$ be a collection of $m$ 
subsets $A\subset\{1,2,\cdots,d\}$ with $|A|\le s$. 
Then, there exists a certain numerical constant $C$ such that 
 for both $\hbSigma = \hbSigmat$ and $\hbSigma = \hbSigmar$, 
 \bel{th2-2}
 && \|(\hbSigma -\bSigma)_{A\times A}\|_{S} 
 + \norm{(\hmatT- \matT)_{A\times A}}{S} 
+ \norm{(\hmatR- \matR)_{A\times A}}{S}
 \cr &\le & C\|\bSigma_{A\times A}\|_S\Big(\sqrt{(s+t+\log m)/n}
+(s+t+\log m)/n\Big)
\cr && +\ C \norm{\bSigma_{A\times A}}{(2,\infty)}
\norm{\bSigma_{A\times A}}{S}^{1/2}\sqrt{(t+\log m)/n} + C s(\log d + t)/n
\eel
simultaneously for all $A\in\scrA_{s,m}$ with at least probability $1-e^{-t}$.
\end{theorem}

\begin{corollary}\label{cor:GenUconcsumm}
If $\,t+\log d \le \beta \max\big\{\log(ed/s),\sqrt{(n/s)(t/s+\log(ed/s))}\big\}$, then 
for both $\hbSigma = \hbSigmat$ and $\hbSigma = \hbSigmar$ 
and a certain numerical constant $C$, 
\bel{eq:corsumm}
 && \max_{|A|\leq s} \frac{\|(\hbSigma -\bSigma)_{A\times A}\|_{S} 
 + \bbE \norm{(\hmatT- \matT)_{A\times A}}{S} 
+ \bbE \norm{(\hmatR- \matR)_{A\times A}}{S}}{
\norm{\bSigma_{A\times A}}{S}
 +\norm{\bSigma_{A\times A}}{S}^{1/2}\norm{\bSigma_{A\times A}}{(2,\infty)}}
 \\ \nonumber  &\le & C(1+\beta)\Big(\sqrt{(t+s\log(ed/s))/n}+(t+s\log(ed/s))/n\Big) 
\eel
with at least probability $1-e^{-t}$.
\end{corollary}

\begin{remark} Corollary \ref{cor:GenUconcsumm} illustrates that 
for $\max_{|A|\leq s}\norm{\bSigma_{A\times A}}{S}=\calO(1)$ and under a mild condition on $(n,d,s)$,  
Theorem~\ref{th:GenUconc2} yields a sparse spectral error bound that matches (\ref{eq:dsz1}) 
of the latent $\tbSigma^s$. 
Note that $\|\bSigma_{A\times A}\|_{(2,\infty)}\le \|\bSigma_{A\times A}\|_{S}$. 
In comparison, the spectral error bounds in \cite{LH13a} and \cite{WZ13}, which apply to the elliptical copula family, 
leads to $\max_{|A|\le s}\norm{(\hbSigmat-\bSigma)_{A\times A}}{S}=\calO(s\sqrt{(\log d)/n})$ by the union bound. 
\cite{LH13a} provided a concentration inequality of order $\sqrt{s(\log d)/n}$ for $\hbSigmat$ 
in the transelliptical family under an additional `sign sub-Gaussian' condition.  
They also provide two examples of elliptical copulas that satisfy the sign sub-Gaussian condition. 
The first example is the case of elliptical copulas with the latent correlation $\bSigma$ 
satisfying a compound symmetric structure (i.e. $\Sigjk=\rho$ for all $j\neq k$). 
The second example is the case when $\bSigma$ has a diagonal block structure with each diagonal block 
having a compound symmetric structure. 
However, it is unclear if the sign sub-Gaussian condition is readily verifiable in general. 
Theorem~\ref{th:GenUconc2} and Corollary~\ref{cor:GenUconcsumm} establish the concentration of 
the nonparametric estimates for the Gaussian copula model without the sign sub-Gaussianity condition, 
although the Gaussian copula family is smaller than the transelliptical family. 
\end{remark}

The corollary below states a simpler but slightly weaker version of 
Theorem \ref{th:GenUconc2} for $s=d$. 
It matches (\ref{eq:dsz1}) for $s=d$ when $\norm{\bSigma}{S}=\calO(1)$ and $t+\log d = \calO(\sqrt{n/d})$. 

\begin{corollary}\label{cor:GenUconc2} 
For a certain numerical constant $C$, 
\bel{eq:corsumm}
\|\hbSigma -\bSigma\|_{S} 
&\leq& C\norm{\bSigma}{S}\Big(\sqrt{(t+d)/n} +(t+d)/n\Big) 
 \cr && + C\norm{\bSigma}{S}^{1/2}\norm{\bSigma}{(2,\infty)}\sqrt{t/n}+ C (t+\log d)d/n
\eel
with at least probability $1-e^{-t}$ for both $\hbSigma = \hbSigmat$ and $\hbSigma = \hbSigmar$.
\end{corollary}

The proof of Theorem \ref{th:GenUconc2} is carried out by establishing large deviation inequalities 
for the first two terms in the decomposition in (\ref{outline-2}), 
an application of Lemma \ref{lm-exp-zero} to the third, and an application of an inequality of \cite{WZ13}
to the fourth. 
 
\begin{lemma}\label{lm-exp-Delta-one-zero}
Let us take $C_1= 2/\pi+1\le 2$ for Kendall's tau and $C_1\le 1+\sqrt{8}/\pi\le 2$ for Spearman's rho. For both Kendall's tau and Spearman's rho, 
\bel{lm-exp-Delta-one-zero-1}
&& \norm{\bDelone-\bDelzero}{S}
\le \sqrt{\frac{C^{2}_{1}\norm{\bSigma}{F}^{2}-2d}{n}} + 2\sqrt{2}\norm{\bSigma}{(2,\infty)}\norm{\bSigma}{S}^{1/2}\sqrt{\dfrac{t}{{n}}}
\eel
with at least probability $1-e^{-t}$. 
\end{lemma}


\begin{lemma}\label{lm-exp-Delta-two}
Let $\matU_{n} - \bbE \matU_{n}-m\bDelone$ be as in (\ref{outline-2}). 
Then, for a certain constant $C$, 
\begin{align*}
\max_{|A|\le s} \norm{\big(\matU_{n} - \bbE \matU_{n}-m\bDelone\big)_{A\times A}}{S} 
\le C s(\log d + t)/n
\end{align*}
with at least probability $1-e^{-t}$.  
\end{lemma}

We state an inequality of \cite{WZ13} in Lemma \ref{lm8} (i) below and its extension 
to Spearman's rho in Lemma \ref{lm8} (ii). 

\begin{lemma} \label{lem:Esigdelt}\label{lm8}
(i) Let $\hbSigmat =(\Sigmahat^\tau_{jk})_{d\times d}$ 
be as in (\ref{eq:sigesttau}) and $\bDeltat = (\Deltajkt)_{d\times d}$ with $\Deltajkt = \htau_{jk}-\tau_{jk}$. 
Let $n_2=2\lfloor n/2\rfloor$ where $\lfloor x \rfloor$ is the integer part of $x$. Then, 
\bel{lem:Esigdelt-1}
\big\|(\hbSigma^{\tau} - \bSigma)_{A\times A}\big\|_S 
\leq \pi \big\|(\hmatT -\matT)_{A\times A}\big\|_S + \frac{s\pi^2}{8}\big\|\bDelta^\tau\big\|_{\max}^2, 
\eel
with \ $\bbP\big\{ \big\|\bDelta^\tau\big\|_{\max} > 2t \big\}\le d^2 e^{- n_2t^2}$ for all $t>0$. \\
(ii) Let $\hbSigmar =(\Sigmahat^\rho_{jk})_{d\times d}$ 
be as in (\ref{eq:sigestrho}) and $\bDeltar = (\Deltajkr)_{d\times d}$ with $\Deltajkr = \hrho_{jk}-\bbE\hrho_{jk}$. 
Let $n_3=3\lfloor n/3\rfloor$ where $\lfloor x \rfloor$ is the integer part of $x$. 
Then, 
\bel{lem:Esigdelr-1}
&& \big\|(\hbSigma^{\rho} - \bSigma)_{A\times A}\big\|_S 
\leq C_2\big\|(\hmatT -\matT)_{A\times A}\big\|_S + \frac{s\pi^2}{36}\big\|\bDelta^\rho\big\|_{\max}^2
+\frac{\pi s^{1/2}\|\bSigma_{A\times A}\|_{(2,\infty)}}{3(n+1)}
\eel
with $C_2=(\pi/3)(2-\sqrt{1-1/4})<1.2$, 
and $\bbP\big\{ \big\|\bDelta^\rho\big\|_{\max} > \sqrt{6}t \big\}\le d^2 e^{- n_3t^2}$ for all $t>0$.
\end{lemma}

\begin{proof}[Proof of Theorem \ref{th:GenUconc2}] 
We consider only $\hbSigmat$ as the case for $\hbSigmar$ is nearly identical.
It follows from Lemma \ref{lm8} that 
\bel{pf-th2-0}
&& \big\|(\hbSigma^{\tau} - \bSigma)_{A\times A}\big\|_S 
\leq \pi \big\|(\hmatT -\matT)_{A\times A}\big\|_S + C s(t+\log d)/n,\quad \forall\ |A|\le s, 
\eel
with at least probability $1-e^{-t}$. 
As in the decomposition in (\ref{outline-2}), 
\bel{outline-3}
&& \hmatT -\matT 
= \left\{\bDeltat - 2\bDelone\right\} + 2\left\{\bDelone-\bDelzero\right\} + 2\bDelzero. 
\eel
It follows from Lemma \ref{lm-exp-Delta-two} that with at least probability $1-e^{-t}$, 
\bel{pf-th2-1}
&& \max_{|A|\le s}\Big\|\left\{\bDeltat-2\bDelone\right\}_{A\times A}\Big\|_{S} 
\le C s(\log d + t)/n. 
\eel
By applying Lemma \ref{lm-exp-Delta-one-zero} to the $m$ sub-matrices with the union bound, 
\bel{pf-th2-2}
&& \norm{(\bDelone-\bDelzero)_{A\times A}}{S}
\le C\norm{\bSigma_{A\times A}}{F}/\sqrt{n}
\cr && \qquad\quad +C\norm{\bSigma_{A\times A}}{(2,\infty)}\norm{\bSigma_{A\times A}}{S}^{1/2}\sqrt{(t+\log m)/n},
\quad \forall\ A\in\scrA_{s,m},\eel
with at least probability $1-m\exp( - t - \log m)\ge 1- e^{-t}$. 
Similarly, Lemma \ref{lm-exp-zero} yields 
\bel{pf-th2-3}
\norm{(\bDelzero)_{A\times A}}{S} 
&\leq& C\|\bSigma_{A\times A}\|_S\sqrt{(s+t+\log m)/n}
\cr && + C\|\bSigma_{A\times A}\|_S(s+t+\log m)/n,\quad \forall\ A\in\scrA_{s,m}
\eel
with at least probability $1-e^{-t}$.  
The first term in (\ref{pf-th2-2}) is dominated by the first term in (\ref{pf-th2-3}) due to 
$\norm{\bSigma_{A\times A}}{F}\le \sqrt{s}\|\bSigma_{A\times A}\|_S$. 
Thus, applying (\ref{pf-th2-1}), (\ref{pf-th2-2}) and (\ref{pf-th2-3}) to (\ref{outline-3}) 
yields (\ref{th2-2}) via (\ref{pf-th2-0}). 
\end{proof}

\section{Discussion} \label{sec:Application}
We describe two applications of our concentration inequality in the $d>n$ case. 


\subsection{Tapering Estimate of Bandable Correlation Matrices}
We consider the Gaussian copula model in (\ref{eq:mod1}). 
We assume that the correlation matrix has a bandable structure in that the off-diagonal 
elements fall off to zero as we move further away from diagonal. 
There are several formulations of such bandability. 
As in \cite{CZZ10}, we consider the parameter class
\begin{align}
 \scrF_{\alpha}(M_{0},M_{1}) & = \left\{\bSigma: 
 \max_{j}\sum_{|i-j|>k}|\Sigma_{ij}| \leq M_{0}k^{-\alpha}\ \forall k, 
 \norm{\bSigma}{S} \leq M_{1}\right\}. 
\end{align}
We adopt the estimator of \cite{CZZ10} and plug in $\hbSigmat$ and $\hbSigmar$: 
\begin{align}\label{def:taper}
  \hbSigmatt = (w_{ij}\hSigma^{\tau}_{ij})_{d\times d} \quad \hbSigmart = (w_{ij}\hSigma^{\rho}_{ij})_{d\times d}
 \end{align}
where $w_{ij}$'s are defined as
\begin{align*}
 w_{ij} = \begin{cases}
           1 & \text{when } |i-j| \leq k/2\\
           2 - 2\dfrac{|i-j|}{k} & \text{when } k/2 < |i-j| < k\\
           0 & \text{otherwise}
          \end{cases}
\end{align*}

The nonparametric tapering estimator $\hbSigmart$ has been considered previously in \cite{XZ12a}, 
where an error bound 
\bes
\sup_{\bSigma\in \scrF_{\alpha}(M_{0},M_{1}) }\bbE_{\bSigma} \Big\| \hbSigmart - \bSigma\Big\|_{S}^2 \le C_{M_0,M_1}\Big(\frac{k^2\log d}{n} + k^{-2\alpha}\Big)
\ees
was established using a generalization of McDiarmid's inequality, where 
$\bbE_{\bSigma}$ is the expectation in the Gaussian copula model (\ref{eq:mod1}) with correlation 
$\bSigma$ in (\ref{eq:corr}), and $C_{M_0,M_1}$ is a constant depending on $M_0$ and $M_1$ only. 
It was mentioned in their paper that the above error bound may not be sharp 
as some key concentration inequalities were not available for rank-based estimators. 
Such key concentration inequalities are provided in Theorem \ref{th:GenUconc2} 
as the rate-optimal error bound in the following theorem demonstrates. 

\begin{theorem}\label{th:tapertaurho} 
Let $\bbE_{\bSigma}$ be the expectation under which (\ref{eq:mod1}) and (\ref{eq:corr}) hold. 
Consider the tapered estimators $\hbSigma_{(k)} = \hbSigmatt$ or $\hbSigma_{(k)}= \hbSigmart$ 
given in (\ref{def:taper}). Then, 
\bel{th3-1}&&
\sup_{\bSigma\in \scrF_{\alpha}(M_{0},M_{1}) }\bbE_{\bSigma} \Big\| \hbSigma_{(k)} - \bSigma\Big\|_{S}^2 \le C_{M_0,M_1}\Big(\frac{k+\log d}{n} + \frac{k^2(\log d)^2}{n^2}+ k^{-2\alpha}\Big)
\eel
for all $1\le k\le n$, where $C_{M_0,M_1}$ is a constant depending on $M_0$ and $M_1$ only. 
In particular, for $k = \min\big(n^{1/(2\alpha+1)},d\big)$ and $\log d\le \beta n^{\alpha/(1+2\alpha)}$, 
\bel{th3-2}
&& \sup_{\bSigma\in \scrF_{\alpha}(M_{0},M_{1}) }\bbE_{\bSigma} \Big\| \hbSigma_{(k)} - \bSigma\Big\|_{S}^2 \le C_{M_0,M_1}(1+\beta)\min\Big(n^{\frac{-2\alpha}{1+2\alpha}}+\frac{\log d}{n},\frac{d}{n}\Big). 
\eel
\end{theorem}

The rate-optimality of (\ref{th3-2}) was proved in \cite{CZZ10} and a combination of their analysis and 
Theorem \ref{th:GenUconc2} proves Theorem \ref{th:tapertaurho}. 
For $\matH = (H_{ij})_{d\times d} = \hbSigma - \bSigma$, 
\bes
(w_{ij}H_{ij})_{d\times d} = k^{-1}\sum_{\ell=1}^{d+2k-1} \matH_{A_\ell\times A_\ell} - 
k^{-1}\sum_{\ell=1}^{d+k-1} \matH_{B_\ell\times B_\ell}
\ees
where $A_\ell = \{1\vee (\ell - 2k),\ldots,\ell\}$ for $1\le \ell < p+2k$ and 
$B_\ell = \{1\vee(\ell - k),\ldots,\ell\}$ for $1\le \ell < p+k$. 
Let $A_{d+2k+\ell-1}=B_\ell$. 
Since $\{\matH_{A_{\ell+2jk} \times A_{\ell+2jk}}, \ell+2jk < d+2k\}$ are disjoint diagonal blocks 
for $\ell=1,\ldots,2k$ and $\{\matH_{A_{\ell+jk} \times A_{\ell+jk}}, \ell+jk \ge d+2k\}$ are disjoint 
diagonal blocks for $\ell=1,\ldots,k$, 
\bes
\Big\|(w_{ij}\hSigma_{ij})_{d\times d} - \bSigma\Big\|_S
\le \Big\|\big((1-w_{ij})\Sigma_{ij}\big)_{d\times d}\Big\|_S
+3\max_{\ell \le 2d+3k-2}\Big\|\matH_{A_\ell\times A_\ell}\Big\|_{S}
\ees
with $|A_\ell|\le 2k$. Since $w_{ij}=0$ for $|i-j|\le k$, the first term above is bounded by 
$M_0k^{-\alpha}$ in the class. 
It follows from Theorem \ref{th:GenUconc2} that the second term above is bounded by 
\bes
\bbE_{\bSigma}\max_{\ell \le 2d+3k-2}\Big\|\matH_{A_\ell\times A_\ell}\Big\|_{S}^2 
\le C_{M_0,M_1}\int_0^\infty\Big(\frac{k+t+\log d}{n}+\frac{k^2(\log d+t)^2}{n^2}\Big)e^{-t}dt, 
\ees
which implies (\ref{th3-1}). 

Although the estimator in (\ref{def:taper}) is not adaptive due to the requirement of $k$ as an input, 
this example demonstrates the utility of our results when Kendall's tau and Spearman's rho 
are used in place of the oracle sample covariance matrix. 
Based on the availability of the latent sample covariance matrix $\hbSigma^s$, \cite{CY12} proposed 
a block thresholding estimator to achieve the optimal rate in (\ref{th3-2}) without the knowledge of $\alpha$. 
An interesting problem is whether the same can be achieved using the Kendall's tau or Spearman's rho, 
as it seems to need a modification of Theorem \ref{th:GenUconc2} for off diagonal blocks of the 
error $\hbSigma-\bSigma$.  

\subsection{Principal Component Analysis} 
Theorem \ref{th:GenUconc} immediately yields the following theorem via 
the \cite{Weyl12} and \cite{DK70} inequalities. 

\begin{theorem}\label{th:modhighpca} 
Consider the Gaussian copula model in (\ref{eq:mod1}). 
Let $\bP_k$, $\hbP_k^\tau$ and $\hbP_k^\rho$ be the projections 
to the span of the $k$ leading eigenvectors of $\bSigma$, $\hbSigmat$ and $\hbSigmar$ 
respectively corresponding to their $k$ largest eigenvalues. 
Let $\lambda_j$ be the $j$-th largest eigenvalue of $\bSigma$. 
Then, for a certain numerical constant $C$, 
\bes
\max\Big(\bbE \left\|\hbP_k^\tau-\bP_k\right\|_S,
\bbE \left\|\hbP_k^\tau-\bP_k\right\|_S\Big)
 \leq C\|\bSigma\|_S(\sqrt{d/n}+d/n)/(\lambda_{k}-\lambda_{k+1}). 
\ees
\end{theorem}

Now we consider the problem of estimating the direction of a sparse leading eigenvector. 
We illustrate the utility of our sparse spectral error bound in the sparse PCA problem by plugging in 
$\{\hbSigmat,\hbSigmar\}$ in place of $\tbSigmas$ in a formulation of \cite{VL12}. 
In particular, we consider an integer $s<d$ to be an upper bound on the number of nonzero components 
of the principal eigenvector $\btheta_1$ of $\bSigma$. 
The following describes the sparse estimates of the principal eigenvector based on $\hbSigmat$ and $\hbSigmar$.
\begin{equation}
 \hbthetaspt = \argmax_{\bv \in \bbS^{d-1}:\norm{\bv}{0} \leq s} \left|\bv^{T}\hbSigmat\bv \right| \qquad \hbthetaspr = \argmax_{\bv \in \bbS^{d-1}:\norm{\bv}{0} \leq s} \left|\bv^{T}\hbSigmar\bv \right|
 \label{eq:spthetahattr}
\end{equation}
The following theorem provides the rate of convergence for sparse PCA. 

\begin{theorem}[Sparse PCA] \label{th:ulthighpca}
Consider the Gaussian copula model in (\ref{eq:mod1}). Let 
$(\lambda_{1}, \btheta_{1})$ be the leading eigenpair of $\bSigma$ with $\|\btheta_1\|_0\le s\to\infty$. 
Let $\lambda_{2}$ be the second largest eigenvalue of $\bSigma$. 
Let $\hbthetaspt$ and $\hbthetaspr$ be the estimate obtained by the optimization defined in (\ref{eq:spthetahattr}). 
If $t+\log d \le \beta\sqrt{(n/s)(t+\log(ed/s))}$, then 
for both $\hbtheta_{1;s}=\hbthetaspt$ and $\hbtheta_{1;s}=\hbthetaspr$ and 
some numeric constant $C>0$, 
\begin{align*}
 & \left|\sin\angle(\hbtheta_{1;s},\btheta_{1})\right| \leq \dfrac{C(1+\beta)}{\lambda_{1}-\lambda_{2}}\
 \Big(\norm{\bSigma}{S}+\norm{\bSigma}{S}^{1/2}\norm{\bSigma}{(2,\infty)}\Big)\sqrt{(t+s\log(ed/s))/n}
 \end{align*}
 with probability at least $1-e^{-t}$.
\end{theorem}

 Theorem \ref{th:ulthighpca} follows from Corollary \ref{cor:GenUconcsumm} by 
 an application of a similar result from \cite{WHL13}. We omit the proofs.

\appendix
\section{Auxiliary Lemmas}\label{app:AuxLemmas}

\begin{proof}[Proof of Lemma \ref{lem:kernel}] By (\ref{eq:tauy}), the kernel for Kendall's tau is 
\bes
h_{j,k}(\bx_1,\bx_2)=\sgn(x_{1j}-x_{2j})\sgn(x_{1k}-x_{2k}). 
\ees
The definition of $\hbar(x,y,\rho)$ in (\ref{hbar}) directly yields 
(\ref{lm-kernel-1}) and the first identity of (\ref{lm-kernel-3}). It remains to verify the 
properties of $g(x,y,\rho)$ in (\ref{lm-kernel-2}) and compute the expectation in (\ref{lm-kernel-3}). 

We first prove the following inequality: 
\begin{align}\label{pf-lm-g-1}
\max_{y} \Big|\Phi(y)-\Phi(y\sqrt{1-\rho^2})\Big| \le |\rho|/2,\ \forall -1\le\rho\le 1. 
\end{align}
For fixed $\rho$, the above maximum is attained, 
$(d/dy)\{\Phi(y) - \Phi(y\sqrt{1-\rho^2})\}=0$, 
when $e^{-y^2/2}=\sqrt{1-\rho^2}e^{-y^2(1-\rho^2)/2}$ or equivalently 
$(1-\rho^2)e^{y^2\rho^2}=1$. Let $y_\rho = \rho^{-1}\sqrt{-\log(1-\rho^2)}$ be the solution. 
Since the equality is attained in (\ref{pf-lm-g-1}) at $\rho=1$, (\ref{pf-lm-g-1}) is a consequence of 
\bel{pf-lm-kernel-1}
&& \frac{d}{d \rho}\frac{\Phi(y_\rho)- \Phi(y_\rho\sqrt{1-\rho^2})}{\rho}
\cr &=& \frac{\varphi(y_\rho\sqrt{1-\rho^2})}{\sqrt{1-\rho^2}} - \frac{\Phi(y_\rho)- \Phi(y_\rho\sqrt{1-\rho^2})}{\rho^2}
\\ \nonumber &\ge& 0. 
\eel
By the monotonicity of the normal density $\varphi(t)$ in $|t|$, 
\bes
\Phi(y_\rho)- \Phi(y_\rho\sqrt{1-\rho^2})\le y_\rho\big(1-\sqrt{1-\rho^2}\big)\varphi(y_\rho\sqrt{1-\rho^2}). 
\ees
Since $y_\rho\rho = \sqrt{-\log(1-\rho^2)}\le \sqrt{\rho^2/(1-\rho^2)}$, (\ref{pf-lm-kernel-1}) follows from 
\bes
y_\rho\big(1-\sqrt{1-\rho^2}\big) = \frac{y_\rho\rho^2}{1+\sqrt{1-\rho^2}}
\le \frac{\rho^2}{\sqrt{1-\rho^2}}. 
\ees
This completes the proof of (\ref{pf-lm-g-1}). 

The joint normal density can be factorized as 
$\varphi_\rho(u,v) = \varphi(u)\varphi_\rho(v|u)$ with the conditional density 
$\varphi_\rho(v|u)\sim N(\rho u,1-\rho^2)$. By (\ref{hbar}), 
\bel{g(x,y)}
g(x,y,\rho) 
&=& \int \sgn(x-u)\varphi(u)\Big\{\int \sgn(y-v)\big\{\varphi_\rho(v|u)-\varphi(v)\big\}dv\Big\}du
\cr &=& \int \sgn(x-u)\varphi(u)\Big\{2\int_{-\infty}^y\big\{\varphi_\rho(v|u)-\varphi(v)\big\}dv\Big\}du
\cr &=& 2\int \sgn(x-u)\varphi(u)\Big\{\Phi((y-\rho u)/\sqrt{1-\rho^2}) - \Phi(y)\Big\}du. 
\eel
This gives the first part of (\ref{lm-kernel-2}) since 
$|\Phi((y-\rho u)/\sqrt{1-\rho^2})-\Phi(y-\rho u)|\le |\rho|/2$ by (\ref{pf-lm-g-1}) 
and $|\Phi(y-\rho u) - \Phi(y)|\le |\rho u|/\sqrt{2\pi}$. 

Similarly, since $\sgn(x-u) = 2I\{u\le x\}-1$, 
\bes
\frac{\pa}{\pa x}g(x,y,\rho) 
&=& \frac{\pa}{\pa x} 4\int_{-\infty}^x \varphi(u)\Big\{\Phi((y-\rho u)/\sqrt{1-\rho^2}) - \Phi(y)\Big\}du
\cr &=& 4\varphi(x)\Big\{\Phi((y-\rho x)/\sqrt{1-\rho^2}) - \Phi(y)\Big\}.
\ees
It follows that 
\bes
\Big|\frac{\pa}{\pa x}g(x,y,\rho)\Big|
&=& 4\varphi(x)\Big|\Phi((y-\rho x)/\sqrt{1-\rho^2}) - \Phi(y)\Big|
\cr &\le& 4\varphi(x)\Big(\frac{|\rho x|}{\sqrt{2\pi}}+\frac{|\rho|}{2}\Big). 
\ees
This gives the second part of (\ref{lm-kernel-2}) due to 
\bes
\max_{x>0}4\varphi(x)(x/\sqrt{2\pi}+1/2)\le 0.987 < 1. 
\ees

For $j\neq k$,  (\ref{pop-rho}) gives 
\bes
\bbE\hbar_0(X_{1j},X_{1k},0)=\bbE\sgn(X_{1j}-X_{2j})\sgn(X_{1k}-X_{3k}) = \rho_{jk}/3.
\ees
Since $U=\Phi(X_1)\sim\ $uniform$(0,1)$, 
$\int \hbar_0^2(x)\varphi(x)dx = 4\Var(U)=1/3$. 
The second  identity of (\ref{lm-kernel-3}) follows. 
\end{proof}

\begin{proof}[Proof of Lemma \ref{lm-kernel-r}] 
We need to include the sample size $n$ in the subscript. 
As in \cite{Hoeff48}, Spearman's rho can be written as 
\begin{equation}
 \hrho_{n,jk} = \dfrac{n-2}{n+1}u_{n,jk}+ \dfrac{3}{n+1}\htau_{n,jk}
 \label{eq:RUT}
\end{equation}
where $u_{n,jk}$ is a U-statistic of order 3 with kernel 
\bel{pf-lm-kernel-a1}
h^*_{jk}(\bx_1,\bx_2,\bx_3) = 3\sgn(x_{1,j}-x_{2j})\sgn(x_{1k}-x_{3k}). 
\eel
For $x\in [0,\pi/2]$, both $\sin x$ and $\sin x - 2\sin(x/3)$ are concave functions with $\sin x - 2\sin(x/3)=0$ at 
the two endpoints, so that $\sin(2x/3) \le 2\sin(x/3) \le \sin x$. Thus, with $x=\pi|\rho_{jk}|/2$, (\ref{pop-eq}) implies that 
\bel{pf-lm-kernel-4}
\sgn(\tau_{jk}) =\sgn(\rho_{jk}),\ (\pi/3)|\rho_{jk}| \le (\pi/2)|\tau_{jk}| \le (\pi/2)|\rho_{jk}|. 
\eel
Since $\bbE u_{jk}=\rho_{jk}$, $|\bbE\hrho_{jk} - \rho_{jk}|= 3|\rho_{jk}-\tau_{jk}|/(n+1)\le |\rho_{jk}|/(n+1)$. 
This gives (\ref{lm-kernel-4}) as $|\rho_{jk}|\le|\Sigma_{jk}|$ by the concavity of $\sin(t)$ in $(0,\pi/6)$.  
Since $u_{n,jk}$ and $\htau_{n,jk}$ are U-statistics with kernel independent of $n$, 
$\hrho_{n,jk}$ is a U-statistic with kernel 
\bel{pf-lm-kernel-a2}
h_{jk}^\rho(\bX_1,\bX_2,\bX_3) 
= \dfrac{n-2}{n+1}u_{3,jk}+ \dfrac{3}{n+1}\htau_{3,jk}. 
\eel
Since $|u_{3,jk}|=|4\hrho_{3,jk}-3\htau_{3,jk}|\le 1$ always holds, 
(\ref{lm-kernel-5}) follows. 

Let $\gbar(x,\rho) = \int \hbar(x,y,\rho)\varphi(y)dy$. It follows from (\ref{pf-lm-kernel-a1}) that 
\bes
\bbE\Big[u_{3,jk}\Big|\bX_1=\bx\Big] = \hbar(x_{j},x_{k},0) + \gbar(x_{j},\Sigma_{jk})+\gbar(x_{k},\Sigma_{jk}). 
\ees
Similarly, 
$\bbE\big[3\htau_{3,jk}\big|\bX_1=\bx\big] = 2\hbar(x_{j},x_{k},\Sigma_{jk})+\tau_{jk}$. Thus, we may take
\bes
\hbar^\rho(x_j,x_k,\Sigma_{jk}) 
&=&  \dfrac{n-2}{n+1}\Big(\hbar(x_{j},x_{k},0) + \gbar(x_{j},\Sigma_{jk})+\gbar(x_{k},\Sigma_{jk})\Big) 
\cr && + \dfrac{2}{n+1}\hbar(x_{j},x_{k},\Sigma_{jk})
\ees
with $c_{jk}=\tau_{jk}/(n+1)$ in (\ref{eq:hjk}). 
Since $\gbar(x,0)=\int \hbar(x,0)\hbar(y,0)\varphi(y)dy=0$, (\ref{lm-kernel-8}) holds. Moreover, 
with $g(x,y,\rho) = \hbar(x,y,\rho) - \hbar(x,y,0)$ as in (\ref{lm-kernel-2}),  
\bes
\hbar^\rho(x,y,\rho)- \hbar^\rho(x,y,0) 
= \dfrac{n-2}{n+1}\Big(\gbar(x,\rho)+\gbar(y,\rho)\Big) + \dfrac{2}{n+1}g(x,y,\rho), 
\ees
so that (\ref{lm-kernel-6}) and (\ref{lm-kernel-7}) are consequences of 
\begin{align}\label{lm-g-2}
\big|\gbar(x,\rho)\big| \le |\rho|\Big(\frac{\sqrt{2}}{\pi}+\frac{1}{2}\Big),\quad 
\Big|\frac{\pa}{\pa x}\gbar(x,\rho)\Big| \le |\rho|, 
\end{align}

Since $\int\sgn(x-u)\varphi(x)dx = - \hbar_0(u)$, (\ref{g(x,y)}) and (\ref{pf-lm-g-1}) yield  
\bes
\Big|\gbar(y,\rho)\Big| 
&=& \Big| 2 \int \hbar_0(u)\varphi(u)\Big\{\Phi((y-\rho u)/\sqrt{1-\rho^2}) - \Phi(y)\Big\}du\Big|
\cr  &\le& 2\int \Big|\hbar_0(u)\big(\rho/2+\rho u/\sqrt{2\pi}\big)\Big|\varphi(u)du 
\ees
Since $\int |\hbar_0(u)|\varphi(u)du = \int_0^1|2x-1|dx =1/2$ and 
\bes
\int |\hbar_0(u)u|\varphi(u)du = -2\int_0^\infty \hbar_0(u)d\varphi(u) = 2\int \varphi^2(u)du = 1/\sqrt{\pi}, 
\ees
we have $|\gbar(y,\rho)|\le |\rho|(1/2+\sqrt{2}/\pi)$. In addition, (\ref{lm-kernel-2}) yields 
\bes
\Big|\frac{\pa}{\pa x}\gbar(x,\rho)\Big|
\le \max_{x,y} \Big|\frac{\pa}{\pa x}g(x,y,\rho)\Big|\le |\rho|. 
\ees
Hence, (\ref{lm-g-2}) holds and the proof is complete. 
\end{proof}

\begin{proof}[Proof of Lemma \ref{lm-exp-sum}] 
By Lemmas \ref{lem:kernel} and \ref{lm-kernel-r}, 
both Kendall's tau and Spearman's rho are U-statistics with kernel bounded by 1, 
so that (\ref{var-Delta-2}) holds. 
By (\ref{lm-kernel-2}) and (\ref{lm-kernel-6}), (\ref{g-bound}) holds, so that 
(\ref{Delta-01-bd}) holds. Since completely degenerate U-statistics of order two or higher 
are orthogonal to U-statistics of order 1, (\ref{var-Delta-2}) and (\ref{Delta-01-bd}) yield
\bes
&&\bbE \norm{(\matU_{n} - \bbE \matU_{n}) - m\bDelzero}{F}^2
\cr &\le& \frac{m(m-1)d(d-1)}{n(n-1)}+m^2\Big(C_1^2\sum_{j\neq k}\frac{\Sigma_{jk}^2}{n}+\frac{4d}{45 n}\Big). 
\ees
Inequality (\ref{lm-exp-sum-1}) 
follows from $C_1^2\ge 2+4/45$ and $\sum_{j\neq k}\Sigma_{jk}^2=\|\bSigma\|_F^2-d$. 
\end{proof}

\begin{proof}[Proof of Lemma \ref{lm-exp-zero}]
Let $N_{\epsa}$ be the largest number of $\eps$-balls one can pack in the $(1+\eps)$-ball centered at the origin 
and $\{\buj, j\le N_{\epsa}\}$ 
be the centers of such $\eps$-balls in one of such configurations. 
From straight forward volume comparison we have 
$N_{\epsa}\eps^d \leq (1+\epsa)^{d}$. 
For each $\bu\in\Sd$, $\|\bu-\buj\|_2\le 2\eps$ for some $j\le N_\eps$, so that 
 \begin{align*}
  \left|\bu^{T}\bDelzero\bu\right| & \leq \left|\buj^{T}\bDelzero\buj\right| + \left|(\bu-\buj)^{T}\bDelzero(\bu+\buj)\right|\\
  & \leq \left|\buj^{T}\bDelzero\buj\right| + 2\epsa(2+2\epsa) \norm{\bDelzero}{S}. 
 \end{align*}
 It follows that 
 \begin{align}
  \norm{\bDelzero}{S} \leq \sup_{j \leq N_{\epsa}} \frac{\big|\buj^{T} \bDelzero \buj\big|}{1-4\epsa(1+\eps)},\quad 
  N_{\epsa} \le (1+1/\eps)^d. 
    \label{eq:specnet}
 \end{align}
 Since $\bX$ has iid $N(0,\bSigma)$ rows, it can be written as 
 $\bX = \bZ\bSigma^{1/2}$ with a standard normal matrix $\bZ\in\Re^{n\times d}$. 
Let $\hbarzero{\matX}$ 
be the $n\times d$ matrix with elements $\hbarzero{X_{ij}}=2\Phi(X_{ij})-1$ and 
\bes
\fu{\bZ} = \norm{\hbarzero{\bZ\bSigma^{1/2}}\bu}{2}/\sqrt{n}. 
\ees 
By (\ref{Delta-zero}), $\bDelzero$ has elements 
$(\bbE_n-\bbE)\hbarzero{x_j}\hbarzero{x_k}$ so that  
\bel{pf-lm-1-1}
\bu^{T} \bDelzero \bu = f_{\bu}^2(\bZ) - \bbE f_{\bu}^2(\bZ). 
\eel
Since $(d/dt)\Phi(t)\le 1/\sqrt{2\pi}$, 
for any $\matV,\matW \in \Re^{n\times d}$ we have 
\begin{align*}
 |\fu{\matV} -\fu{\matW}| \leq \sqrt{\dfrac{2}{n\pi}} \norm{(\matV-\matW)\bSigma^{1/2}}{F}
\leq \sqrt{\dfrac{2\norm{\bSigma}{S}}{n\pi}} \norm{\matV-\matW}{F}\end{align*}
Thus, the Lipschitz norm of $\fu{\cdot}$ is bounded by $\sqrt{2\norm{\bSigma}{S}/(n\pi)}$. 
By the Gaussian concentration inequality \citep{Borell75}, we have 
\begin{align}
 \bbP\Big\{ \Big|f_{\bu}(\bZ) - \bbE f_{\bu}(\bZ)\Big| > t \sqrt{2\|\bSigma\|_{S}/(\pi n)}\Big\} \le 2e^{-t^2/2}. 
 \label{eq:h0lipconc1}
\end{align}
It follows that 
\begin{align*}
\bbE f_{\bu}^2(\bZ) - \Big(\bbE f_{\bu}(\bZ)\Big)^2 & = \Var\Big(f_{\bu}(\matX)\Big) \le \frac{2\|\bSigma\|_S}{\pi n}\int_0^\infty e^{-t^2/2}dt^2 
 = \frac{4\|\bSigma\|_S}{\pi n}. 
\end{align*}
We note that $\bbE f_{\bu}^2(\bZ)=\bu^T\matR\bu/3\le \norm{\matR}{S}/3$ as in (\ref{Delta-zero2}), so that
by (\ref{pf-lm-1-1}) 
\begin{align*}
\big|\bu^T \bDelzero \bu\big| & \leq \Big| f_{\bu}^2(\matX) - \Big(\bbE f_{\bu}(\matX)\Big)^2\Big| + \frac{4\|\bSigma\|_S}{\pi n}\\
& \leq \Big( f_{\bu}(\matX) - \bbE f_{\bu}(\matX)\Big)^2 + 2\Big(\|\matR\|_S/3\Big)^{1/2}\Big| f_{\bu}(\matX) - \bbE f_{\bu}(\matX)\Big| + \frac{4\|\bSigma\|_S}{\pi n}. 
\end{align*}
This inequality and (\ref{eq:specnet}) yield 
\begin{align}
\|\bDelzero\|_{S} \le \frac{\zeta_{n}^2 +2\Big(\|\matR\|_S/3\Big)^{1/2}\zeta_{n}+ 4\|\bSigma\|_S/(\pi n)}
{1-4\epsa(1+\epsa)}
\label{eq:Edelzero2}
\end{align}
with $\zeta_{n} = \max_{j\le (1+1/\eps)^d}\Big| f_{\buj}(\matX) - \bbE f_{\buj}(\matX)\Big|$. 
It follows from (\ref{eq:h0lipconc1}) that 
\bel{pf-lm-1-2}
&& \bbP\Big\{ \zeta_{n} > t \sqrt{2\|\bSigma\|_S/(\pi n)}\Big\}\le 2(1+1/\epsa)^d e^{-t^2/2}. 
\eel
Let $x_{*} = 2(d\log(1+1/\eps)+\log 2)$. We have 
\begin{align*}
 \bbE \zeta_{n}^2  & \leq  \frac{2\|\bSigma\|_S}{\pi n}\int_0^\infty \min\Big\{2(1+1/\eps)^d e^{-t^2/2},1\Big\}dt^2
= \frac{2\|\bSigma\|_S}{\pi n}\big(x_{*}+2\big)
\end{align*}
Taking $\epsa$ satisfying $\epsa(1+\eps)=1/20$, we find 
$1/(1-4\eps(1+\eps))=5/4$ and $\log(1+1/\eps)\le \pi$, 
so that $x_*\le 2(\pi d + \log 2)$ and 
\bes
\bbE \zeta_{n}^2\le 4\|\bSigma\|_S\big(d/n + (1+\log 2)/(\pi n)\big).
\ees
Combining this with (\ref{eq:Edelzero2}), we have 
\bes
\bbE\|\bDelzero\|_{S} 
&\le& (5/4)\big\{\bbE \zeta_{n}^2 +2(\|\matR\|_S/3)^{1/2}\bbE\zeta_{n}+ 4\|\bSigma\|_S/(\pi n)\big\}
\cr &\le& 5\|\bSigma\|_S\big\{d/n + (2+\log 2)/(\pi n)\big\}
\cr && +5\big(\|\bSigma\|_S\|\matR\|_S/3\big)^{1/2}\big(d/n+(1+\log 2)/(\pi n)\big)^{1/2}. 
\ees
This yields (\ref{lm-exp-zero-1}) due to $2+\log 2\le \pi$ and $\|\matR\|_S\le\|\bSigma\|_S$. Moreover, by (\ref{pf-lm-1-2})
\bes
&& \bbP\Big\{ \zeta_{n} > \sqrt{2\pi d+2t^2}\sqrt{2\|\bSigma\|_S/(\pi n)}\Big\}\le 2 e^{\pi d-(2\pi d+2t^2)/2}=2e^{-t^2}
\ees
and outside this event (\ref{eq:Edelzero2}) gives 
\bes
\|\bDelzero\|_{S} &\le & 5\|\bSigma\|_S(d/n+(t^2+1)/(\pi n)) 
\cr && +5\big(\|\bSigma\|_S\|\matR\|_S/3\big)^{1/2}\sqrt{d/n+t^2/(\pi n)}. 
\ees
This completes the proof due to $\|\matR\|_S\le\|\bSigma\|_S$. 
\end{proof}

\begin{proof}[Proof of Lemma \ref{lm5}] 
(i) Let $x=(\pi/2)\tau_{jk}$ and $y =(\pi/2)\Deltajkt$ so that 
$\hSigma_{jk}=\sin(x+y)$ and $\Sigma_{jk}=\sin x$. 
Because $\sin(x+y)-\sin x - y = (\cos x-1)y - \int_0^y (y-t)\sin(x+t) dt$, 
\bes
\Big|\hSigma^{\tau}_{jk} -\Sigjk -(\pi/2)\Deltajkt\Big|
\le \frac{2|xy|}{\pi} + \frac{y^2}{2}
\le \frac{\pi}{2}\Big|\tau_{jk}\Deltajkt\Big| + \frac{\pi^2}{8}\Big|\Deltajkt\Big|^2. 
\ees
Since $\htau_{jk}$ is a U-statistic of order $m=2$ and a sign kernel in (\ref{eq:tauy}), 
the Hoeffding decoupling argument gives 
$\bbE(\Deltajkt)^2\le \bbE\big(2\,\hbox{Bin}(n_2,p_{jk})/n_2-2 p_{jk}\big)^2\le 1/n_2$ and 
\bes
\bbE (\Deltajkt)^4 \le \bbE\Big(2\,\hbox{Bin}(n_2,p_{jk})/n_2-2 p_{jk}\Big)^4 \le 3/n_2^2, 
\ees
where $p_{jk}=(1+\tau_{jk})/2$. 
Since $\sum_{j\neq k}\tau_{jk}^2\le \sum_{j\neq k}\Sigma_{jk}^2=\|\bSigma\|_F^2-d$, we have 
\bes
\sum_{j,k} \bbE \Big|\tau_{jk}\Deltajkt\Big|^2\le \frac{\|\bSigma\|_F^2-d}{n_2},\quad  
\sum_{j,k} \bbE  \Big|\Deltajkt\Big|^4\le \frac{3d^2}{n_2^2}. 
\ees
Consequently, (\ref{lm5-1}) holds. 

(ii) Let $x=(\pi/6)\bbE\hrho_{jk}$, $y =(\pi/6)\Deltajkr$ and $z = (\pi/6)(\bbE\hrho_{jk}-\rho_{jk})$ so that 
$\hSigma_{jk}=2\sin(x+y)$ and $\Sigma_{jk}=2\sin(x-z)$. 
Due to $|z|\le (\pi/6)|\Sigma_{jk}|/(n+1)$ by (\ref{lm-kernel-4}), 
\bes
\left |\hSigjkr -\Sigma_{jk} - \frac{\pi}{3}\Delta^{\rho}_{jk}\right| 
&=& 2\Big|\sin(x+y) - \sin(x-z) - y\Big|
\cr &\leq& \frac{4|xy|}{\pi} + y^2+ 2|z|
\\ \nonumber &\le &\frac{\pi}{9}\Big|\Sigma_{jk}\Deltajkr\Big| + \frac{\pi^{2}}{36}|\Deltajkr|^{2}
+\frac{\pi |\Sigma_{jk}|}{3(n+1)}. 
\ees
Similar to part (i), (\ref{lm5-2}) follows from $\bbE(\Deltajkt)^2\le 1/n_3$ and  
$\bbE(\Deltajkt)^4\le 3/n_3^2$. \end{proof}

\begin{proof}[Proof Of Lemma \ref{lm-exp-Delta-one-zero}] 
We write 
\bes
\bDelone-\bDelzero = (\bbEn-\bbE)\matG = n^{-1}\sum_{i=1}^n \matG(\bX_i)-\bbE \matG(\bX_1)
\ees
with $\matG(\bx) = \big(g_{jk}(\bx)\big)_{d\times d}$, 
where $g_{jk}(\bx) = \hbar^\rho(x_{j},x_{k},\Sigjk) - \hbar^\rho(x_{j},x_{k},0)$ for Kendall's tau and 
$g_{jk}(\bx) = \hbar(x_{j},x_{k},\Sigjk) - \hbar(x_{j},x_{k},0)$ for Spearman's rho. 
It follows from (\ref{lm-kernel-2}) and (\ref{lm-kernel-7}) that 
 \begin{align*}
  |g_{jk}(\by) - g_{jk}(\bx) | \leq |\Sigjk| \left\{|y_{j}-x_{j}| + |y_{k}-x_{k}|\right\}. 
 \end{align*}
This inequality implies that for all $d$-dimensional vectors $\bx$ and $\by$, 
\begin{align*}
 \norm{\matG(\bx)-\matG(\by)}{S} 
 & \leq \max_{\bu:\norm{\bu}{2}=1}\sum_{j=1}^{d}\sum_{k=1}^{d}|u_{j}u_{k}||g_{jk}(\bx)-g_{jk}(\by)|\\
 & \leq  \max_{\bu:\norm{\bu}{2}=1}\sum_{j=1}^{d}\sum_{k=1}^{d}|u_{j}u_{k}\Sigma_{jk}|\left( |x_{j}-y_{j}| + |x_{k}-y_{k}|\right)\\
 & \leq 2 \max_{\bu:\norm{\bu}{2}=1}\sum_{j=1}^{d}\sum_{k=1}^{d}|u_{j}u_{k}\Sigma_{jk}(x_{j}-y_{j})|\\
 & \leq 2 \max_{\bu:\norm{\bu}{2}=1} \sum_{j=1}^{d}|u_{j}(x_{j}-y_{j})|\max_{j}\sum_{k}|u_{k}\Sigma_{jk}|\\
 & \leq 2 \norm{\bSigma}{(2,\infty)}\norm{\bx-\by}{2}. 
\end{align*}
Recall that $\matX = (\bX_{1},\cdots,\bX_{n})^{T} \in \Re^{n\times d}$ with iid $\bX_i\sim N(0,\bSigma)$, 
so that the matrix $\bZ = \bX\bSigma^{-1/2}$ has iid $N(0,1)$ entries. 
Since $\bX_{i}$ are iid vectors, we may write $\matM_{G}= \bbE\matG(\bX_1)$. 
Let $\bZ_i = \bSigma^{-1/2}\bX_i$. Define a function $f:\Re^{n\times d} \rightarrow \Re$ by
\begin{align*}
 f(\matZ) = \norm{(\bbEn-\bbE)\matG}{S} 
 = \Big\| \dfrac{1}{n}\sum^{n}_{i=1}\left\{\matG(\bSigma^{1/2}\bZ_i)-\matM_{G}\right\}\Big\|_{S}. 
\end{align*}
For matrices $\matV=(\bV_{1},\cdots,\bV_{n})^{T}$ and $\matW=(\bW_{1},\cdots,\bW_{n})^{T}$ in $\Re^{n\times d}$, 
we have
\begin{align*} 
\left|f(\matV)-f(\matW)\right|  & =  \left|\ \norm{\dfrac{1}{n}\sum^{n}_{i=1}\matG(\bSigma^{1/2}\bV_{i}) - \matM_{G}}{S} - \norm{\dfrac{1}{n}\sum^{n}_{i=1}\matG(\bSigma^{1/2}\bW_{i}) - \matM_{G}}{S}\right|\\
&  \leq \dfrac{1}{n}\sum^{n}_{i=1}\norm{\matG(\bSigma^{1/2}\bV_{i}) - \matG(\bSigma^{1/2}\bW_{i})}{S} \\
& \leq 2\norm{\bSigma}{(2,\infty)}\dfrac{1}{n}\sum^{n}_{i=1}\norm{\bSigma^{1/2}\bV_{i}-\bSigma^{1/2}\bW_{i}}{2}\\
& \leq 2\dfrac{\norm{\bSigma}{(2,\infty)}\norm{\bSigma}{S}^{1/2}}{\sqrt{n}}\norm{\matV-\matW}{F}. 
\end{align*}
We have here a Lipschitz continuity in $nd$ variables. 
An application of the concentration inequality for Lipschitz continuous functions yields that for any $t >0$
\begin{equation*}
\bbP\left(f(\matZ) - \bbE f(\matZ) > 2\norm{\bSigma}{(2,\infty)}\norm{\bSigma}{S}^{1/2}\dfrac{t}{\sqrt{n}}\right) 
\leq \exp\left\{-t^{2}/2\right\}
\end{equation*}
with $f(\matZ) =\norm{(\bbEn-\bbE)\matG}{S}  = \|\bDelone-\bDelzero\|_{S}$. 
From (\ref{Delta-01-bd}) it follows that
\begin{align*}
\bbE^{2}f(\matZ) \leq \bbE \norm{\bDelone-\bDelzero}{S}^{2} \leq   C_1^2\sum_{j\neq k}\frac{\Sigma_{jk}^2}{n}+\frac{4d}{45 n} \leq \dfrac{C^{2}_{1}\norm{\bSigma}{F}^{2}-2d}{n}, 
\end{align*}
where $C_1= 2/\pi+1\le 2$ for Kendall's tau and $C_1\le 1+\sqrt{8}/\pi\le 2$ for Spearman's rho, 
with $C^{2}_{1}\geq 2+4/45$. 
\end{proof}

\begin{proof}[Proof of Lemma \ref{lm-exp-Delta-two}] 
By Lemmas \ref{lem:kernel} and \ref{lm-kernel-r}, 
$\big(\matU_{n} - \bbE \matU_{n}\big)_{jk}$ are U-statistics of order $m$ 
and their kernels are uniformly bounded by 1, where $m=2$ for Kendall's tau and $m=3$ for Spearman's rho. 
Let $\matD=(D_{jk})_{d\times d}$ with $D_{jk}= \big(\matU_{n} - \bbE \matU_{n}-m\bDelone\big)_{jk}$. 
Since $m\bDelone$ is the first order Hoeffding decomposition of 
$\big(\matU_{n} - \bbE \matU_{n}\big)_{jk}$, 
$D_{jk}$ is second order degenerate. 
Thus, by \cite{AG93}, $\bbP\big\{ \big|D_{jk}\big| > Ct/n\big\} \le 4e^{-t}$ for a certain numerical constant $C$. 
This gives 
$\bbP\Big\{ \norm{\matD}{\max} > Ct/n\Big\} \le 4d^2e^{-t}$. 
Because $\max_{|A|\leq s} \norm{\matD_{A\times A}}{S}\le s\norm{\matD}{\max}$, 
choosing $t = s(2\log 2d+t)$ completes the proof.
\end{proof} 

\begin{proof}[Proof of Lemma \ref{lm8}] 
We prove part (ii) only as part (i) can be found in \cite{WZ13}. 
Let $x=(\pi/6)\bbE\hrho_{jk}$, $y =(\pi/6)\Deltajkr$ and $z = (\pi/6)(\bbE\hrho_{jk}-\rho_{jk})$, 
so that $\hSigma_{jk}=2\sin(x+y)$ and $\Sigma_{jk}=2\sin(x-z)$. 
By (\ref{lm-kernel-4}), 
\bes
&& \left |\hSigjkr -\Sigma_{jk} - \cos((\pi/6)\rho_{jk})(\pi/3)\Deltajkr\right| 
\cr &=& 2\Big|\sin(x+y) - \sin(x-z) - y\cos(x-z)\Big|
\cr &\leq& 2|z| + y^2 
\\ \nonumber &\le & \frac{\pi^{2}}{36}|\Deltajkr|^{2}+\frac{\pi |\Sigma_{jk}|}{3(n+1)}. 
\ees
We have $\|(|\Deltajkr|^{2})_{A\times A}\|_S\le s\big\|\bDelta^\rho\big\|_{\max}^2$ and 
$\|(|\Sigma_{jk}|)_{A\times A}\|_S\le \sqrt{s}\|\bSigma\|_{(2,\infty)}$.  
The tail probability bound for $\big\|\bDelta^\rho\big\|_{\max}$ follows 
by applying the union bound to the \cite{Hoeff63} inequality. 
As in \cite{WZ13}, due to $\cos((\pi/6)\rho_{jk})
=\sqrt{1-\Sigma_{jk}^2/4}$, 
\bes
\Big\|\Big(\cos((\pi/6)\rho_{jk})\Deltajkr\Big)_{A\times A}\Big\|_S
\le \sum_{m=0}^\infty\left|{1/2\choose m}\right|4^{-m}\Big\|\bDeltar\Big\|_S. 
\ees
This completes the proof as $\sum_{m=0}^\infty\left|{1/2\choose m}\right|4^{-m}=2-\sqrt{1-1/4}$. 
\end{proof}

\bibliographystyle{imsart-nameyear}
\bibliography{NonParamCorMat}

\end{document}